\documentclass[12pt, twoside]{article}
\usepackage{amssymb}
\usepackage{amsmath}
\usepackage{cases}
\usepackage{txfonts}
\usepackage{amsfonts}
\usepackage{amsmath,amssymb}
\usepackage{multicol}
\usepackage{color}

\def\Z{\mathbb{Z}}
\def\R{\mathbb{R}}
\def\N{\mathbb{N}}
\def\T{\mathbb{T}}
\def\epsilon{\varepsilon}
\newcommand{\be}{\begin{equation}}
\newcommand{\ee}{\end{equation}}
\newcommand{\baa}{\begin{array}}
\newcommand{\eaa}{\end{array}}
\newcommand{\ba}{\begin{eqnarray}}
\newcommand{\ea}{\end{eqnarray}}

\newtheorem{theorem}{Theorem}[section]
\newtheorem{lemma}[theorem]{Lemma}

\newtheorem{definition}[theorem]{Definition}
\newtheorem{remark}[theorem]{Remark}

\numberwithin{equation}{section}
\newenvironment{proof}[1][Proof]{\noindent\textbf{#1.} }{\hfill $\Box$}
\allowdisplaybreaks

\makeatletter
\begin{document}
\date{}
\title{\bf{Propagating speeds of bistable transition fronts in spatially periodic media}}
\author{Hongjun GUO \thanks{The author was supported by the China Scholarship Council for 3 years of study at Aix Marseille Universit\'e.}\\
\\
\footnotesize{Aix Marseille Univ, CNRS, Centrale Marseille, I2M, Marseille, France}}
\maketitle

\begin{abstract}
\noindent{This paper is concerned with the propagating speeds of transition fronts in $\R^N$ for spatially periodic bistable reaction-diffusion equations. The notion of transition fronts generalizes the standard notions of traveling fronts. Under the a priori assumption that there exist pulsating fronts for every direction $e$ with nonzero speeds, we show some continuity and differentiability properties of the front speeds and profiles with respect to the direction~$e$. Finally, we prove that the propagating speed of any transition front is larger than the infimum of speeds of pulsating fronts and less than the supremum of speeds of pulsating fronts.}

\end{abstract}


\section{Introduction}
\noindent
In this paper, we study the propagating speeds of transition fronts of spatially periodic reaction-diffusion equations of the type
\begin{equation}\label{eq1.1}
u_t=\Delta u+f(x,u),\quad\,(t,x)\in\R\times\R^N,
\end{equation}
where $u_t=\frac{\partial u}{\partial t}$ and $\Delta$ denotes the Laplace operator with respect to the space variables $x\in\R^N$.

Throughout this paper, we assume that the reaction term $f(x,u)$ is $\Z^N$-periodic with respect to $x$. To be more precise, we denote by $\T^N=\R^N/\Z^N$ the $N$-dimensional torus. We assume that the function $f:\T^N\times \R\rightarrow \R$ is continuous, $C^{\alpha}$ in $x$ uniformly with respect to $u\in\R$ with $\alpha\in (0,1)$, of the class $C^2$ in $u$ uniformly with respect to $x\in\T^N$ while the partial derivatives $f_u(x,u)=\partial_u f(x,u)$, $f_{uu}(x,u)=\partial_{uu} f(x,u)$ are Lipschitz continuous in $u$, on $\T^N\times\R$. Moreover, we assume that, for every $x\in\R^N$, the profile $f(x,\cdot)$ is bistable in $[0,1]$, that is, there is $\theta_x\in(0,1)$ such that
\begin{align}\label{F1}
f(x,0)=f(x,1)=f(x,\theta_x)=0,\quad f(x,\cdot)<0\,\text{ on }\, (0,\theta_x),\quad f(x,\cdot)>0\,\text{ on }\, (\theta_x,1).
\end{align}
We also assume that $0$ and $1$ are uniformly (in $x$) stable zeroes of $f(x,\cdot)$, in the sense that there exist $\gamma>0$ and $\sigma\in(0,1/2)$ such that
\begin{equation}\label{F2}
-f_u(x,u)\ge \gamma\,\text{ for all $(x,u)\in\R^N\times[0,\sigma]$ and $(x,u)\in\R^N\times[1-\sigma,1]$}.
\end{equation}
Notice that this implies in particular that $\sigma<\theta_x<1-\sigma$. For mathematical convenience, we assume that $f(x,u)=f_u(x,0)u$ for $(x,u)\in\R^N\times(-\infty,-u_0)$ and $f(x,u)=f_u(x,1)(u-1)$ for $(x,u)\in \R^N\times(1+u_0,+\infty)$ for some positive $u_0$, $-f_u(x,u)\ge \gamma$ for all  $(x,u)\in\R^N\times(-\infty,\sigma]$ and $(x,u)\in\R^N\times[1-\sigma,+\infty)$ and $f(x,u)$, $f_u(x,u)$, $f_{uu}(x,u)$ are globally Lipschitz-continuous in $u$ uniformly in $x\in\R^N$.

The cubic nonlinearity is a typical case of such a function $f$ satisfying \eqref{F1} and \eqref{F2}, that is,
\begin{align}\label{theta}
f(x,u)=u(1-u)(u-\theta_x),
\end{align}
where $0<\theta_x<1$ is a $\Z^N$-periodic $C^{\alpha}(\R^N)$ function with respect to $x$. Moreover, the intermediate zero $\theta_x$ of $f(x,\cdot)$ in \eqref{theta} or more generally in \eqref{F1} is not assumed to be constant in general.

Our main purpose in this paper is to study the propagating speeds of transition fronts which are some classical solutions connecting the two stable states $0$ and $1$. A standard group of transition fronts are so-called pulsating, or periodic fronts for our spatially periodic reaction-diffusion equations. Let us recall the definition of a pulsating front which can be referred to \cite{SKT,X1,X2,X3}.

\begin{definition}[Pulsating fronts]\label{PF}
A pair $(U_e,c_e)$ with $U_e:\R\times\T^N\rightarrow \R$ and $c_e\in\R$ is said to be a pulsating front of \eqref{eq1.1} with effective speed $c_e$ in the direction $e\in\mathbb{S}^{N-1}$ connecting $0$ and $1$ if the two following conditions are satisfied:
\begin{itemize}
\item[(i)] The map $u(t,x):=U_e(x\cdot e-c _e t,x)$ is an entire (classical) solution of the parabolic equation \eqref{eq1.1}.
\item[(ii)] The profile $U_e$ satisfies
$$\lim_{\xi\rightarrow +\infty} U_e(\xi,y)=0,\ \lim_{\xi\rightarrow -\infty} U_e(\xi,y)=1,\,\text{ uniformly for $y\in\T^N$}.$$
\end{itemize}
\end{definition}

Notice that if $(U_e(\xi,y),c_e)$ is a pulsating front of \eqref{eq1.1} in the direction $e\in\mathbb{S}^{N-1}$, then it satisfies the limit condition (ii) in the above definition as well as, if $c_e\neq 0$, the semi-linear elliptic degenerate equation
\begin{align}\label{Ue}
c_e \partial_{\xi} U_e +\partial_{\xi\xi} U_e +2\nabla_y \partial_{\xi} U_e\cdot e+\Delta_y U_e +f(y,U_e)=0,\,\text{ for all $(\xi,y)\in\R\times\T^N$}.
\end{align}

Note that the notion of pulsating front with nonzero speed was first given in \cite{SKT} and further developed in \cite{BH0,X1,X2,X3}. According to these references, it is said that an entire solution $u(t,x)$ of \eqref{eq1.1} is called a pulsating  traveling wave solution in the direction $e\in\mathbb{S}^{N-1}$ and effective speed $c\neq 0$ if it satisfies the following two conditions
\begin{itemize}
\item[(i)] $u(t+\frac{k\cdot e}{c},x)=u(t,x-k)$, for all $k\in\Z^N$ and $(t,x)\in\R\times\R^N$,
\item[(ii)] $\lim_{r\rightarrow +\infty} u(t,r e+y)=0$, $\lim_{r\rightarrow -\infty} u(t,r e+y)=1$, for all $t\in\R$ and $y\in\R^N$.
\end{itemize}
Notice that when the effective speed is nontrivial, this definition is equivalent to Definition~\ref{PF}. In fact, if $(U_e,c_e)$ is a pulsating front with $c_e\neq 0$ in sense of Definition \ref{PF}, $u(t,x)=U(x\cdot e-c_e t,x)$ becomes a pulsating front in sense of \cite{BH0,SKT,X1,X2,X3}. Conversely if $u(t,x)$ is a pulsating front in the direction $e\in\mathbb{S}^{N-1}$ and the effective speed $c\neq 0$, then so is $U(\xi,x):=u(\frac{x\cdot e-\xi}{c},x)$ in the sense of Definition \ref{PF} with $c_e=c$.

Now we review some known existence results on standard traveling waves. In homogeneous case, Aronson and Weinberger \cite{AW} and Fife and Mcleod \cite{FM} have studied the existence and nonexistence of traveling fronts $\phi(x-ct)$ for one-dimensional equation
$$u_t-u_{xx}=f(u)$$
where $f$ is bistable. Especially, if $f$ simply satisfies $f(0)=f(1)=0$, $f<0$ on $(0,\theta)$ and $f>0$ on $(\theta,1)$, it is known to exist a traveling front $\phi(x-ct)$ satisfying
\begin{eqnarray*}
\left\{\begin{aligned}
&\phi''+c\phi'+f(\phi)=0 \text{ in $\R$},\\
&0<\phi<1 \text{ in $\R$, }\,\phi(-\infty)=1\,\text{ and }\,\phi(+\infty)=0.
\end{aligned}
\right.
\end{eqnarray*}
Notice that the propagating speed $c$ has the sign of $\int_0^1 f(u)du$ and the profile $\phi$ is unique up to shifts. For higher dimensions $N\geq 2$, an immediate extension of one-dimensional traveling fronts consists in planar traveling fronts
$$u(t,x)=\phi(x\cdot e-c t)$$
for any given unit vector $e$ of $\mathbb{R}^N$, where $(c,\phi)$ are as above. We denote the level sets by $\{x\in\R^N; u(t,x)=r\}$ for $0<r<1$ and any $t\in\R$. Then, the level sets of planar fronts are parallel hyperplanes which are orthogonal to the propagating direction $e$. We also notice that the profiles of these fronts are invariant as they propagate with speed $c$ in the direction $e$. The existence and uniqueness of these fronts can be referred to the one-dimensional traveling fronts. Besides, in $\R^N$ with $N\ge2$, more general traveling fronts exist, which have non-planar level sets. For instance, conical-shaped axisymmetric non-planar fronts are known to exist for some $f$, see \cite{CGHNR,HMR1,NT1}. Fronts with non-axisymmetric shapes, such as pyramidal fronts, are also known to exist, see \cite{T1,T3}. For qualitative properties of these traveling fronts, we refer to \cite{HM,HMR1,HMR2,NT1,NT2,RRM,T2,T3}.

For explicit spatially periodic dependence, only few results has been obtained in the bistable case. We may refer to the works of Xin \cite{X1,X2,X3} who used refined perturbation arguments to obtain the existence of waves for such periodic equations
\be\label{1.6}
u_t=\sum_i (a(x)u_{x_i})_{x_i} +\sum_i b_i(x)u_{x_i}+f(x,u)
\ee
when the diffusivity matrix $a$ is close to identity and $f$ is independent of $x$. For one dimensional case of \eqref{1.6} when $f(x,u)=g(x)f(u)$ with $0<g_1\le g\le g_2<+\infty$ in $\R$ and $\int_0^1 \min_{[0,1]} f(\cdot,u)du>0$, Nolen and Ryzhik \cite{NR4} proved the existence of pulsating fronts with nonzero speed. Furthermore, if the solutions of \eqref{1.6} with some compactly supported initial conditions can converge locally uniformly to $1$ as $t\rightarrow +\infty$, there exist pulsating fronts with a positive speed for \eqref{1.6}, see \cite{DGM}. Ding et al \cite{DHZ} also obtained some existence results of pulsating fronts for one-dimensional reaction-diffusion equations in a periodic habitat. More precisely, they proved that pulsating fronts exist for small period and large period by applying the implicit function theorem and abstract results of Fang and Zhao \cite{FZ} and they got that the speed has the sign of $\int_{\T^N\times[0,1]} f(x,u)dxdu$ when the speed is not zero. For one dimensional \eqref{eq1.1} with spatially inhomogeneous mixed bistable-ignition reactions, Zlato{\v{s}} \cite{Z3} proved that there exists a unique, up to shifts, right-facing (or left-facing) transition front which is increasing in time. Meantime, he found a periodic pure bistable reaction such that there is no transition front of \eqref{eq1.1}. Thus, pulsating fronts with nonzero speed do not exist in general, we also refer to \cite{DHZ,X4,XZ}.

Throughout this paper, we assume that
\begin{itemize}
  \item[(A1)] $\int_{\T^N\times[0,1]} f(x,u)dxdu\neq 0$,
  \item[(A2)] for any direction $e\in\mathbb{S}^{N-1}$, there is a pulsating front $(U_e,c_e)$ with $c_e\neq 0$ satisfying Definition \ref{PF}.
\end{itemize}
From the result of Ducrot \cite{D} and our Lemma \ref{lemma2.2} in Section 2, it follows that the speed $c_e$ for each direction $e\in\mathbb{S}^{N-1}$ has the sign of $\int_{\T^N\times[0,1]} f(x,u)dxdu$ once the assumptions (A1), (A2) hold. Thus, without loss of generality, one can assume that $\int_{\T^N\times[0,1]} f(x,u)dxdu>0$, that is, $c_e>0$ for all $e\in\mathbb{S}^{N-1}$. In fact, if $c_e<0$ for all $e\in\mathbb{S}^{N-1}$, one can replace $u$, $f$, $U_e(\xi,y)$ by $\tilde{u}=1-u$, $g(x,u)=-f(x,1-u)$, $\tilde{U}_e(\xi,y)=1-U_e(-\xi,y)$ and then, the new pulsating front~$\tilde{U}_e$ propagates with speed $-c_e>0$. From \cite{BH2} and Lemmas \ref{lemma2.1} and \ref{lemma2.4} in Section 2, for any direction $e\in\mathbb{S}^{N-1}$, the speed $c_e$ is then unique and the pulsating front $U_e$ is then unique up to shifts in time.

As we emphasized, even for homogeneous case, there are many types of traveling fronts in higher dimension such as standard planar fronts, conical-shaped axisymmetric non-planar fronts, pyramidal fronts and so on. More complicated structured fronts exist for spatially periodic reaction-diffusion equations. A one-dimensional example can be refer to \cite{DHZ1}, in which the authors established a new type of transition fronts which are not pulsating fronts. Even if the types of traveling fronts are various, there are some common properties shared by them. For all of them, the solutions $u$ converge to the stable states $0$ or $1$ far away from their moving or stationary level sets, uniformly in time. This fact led to the introduction of a more general notion of traveling fronts, that is, transition fronts, see \cite{BH1,BH2} and see~\cite{S} in the one-dimensional setting. In order to recall the notion of transition fronts, let us introduce a few notations. First, for any two subsets $A$ and $B$ of $\mathbb{R}^N$ and for $x\in\R^N$, we set
\begin{equation*}
d(A,B)=\inf\big\{|x-y|;\ (x,y)\in A\times B\big\}
\end{equation*}
and $d(x,A)=d(\{x\},A)$, where $|\cdot|$ is the Euclidean norm in $\mathbb{R}^N$. Consider two families $(\Omega_t^-)_{t\in \mathbb{R}}$ and $(\Omega_t^+)_{t\in \mathbb{R}}$ of open nonempty subsets of $\mathbb{R}^N$ such that
\begin{eqnarray}\label{eq1.3}
\forall t\in \mathbb{R},\ \ \left\{
\begin{aligned}
&\Omega_t^-\cap \Omega_t^+=\emptyset,\\
&\partial \Omega_t^-=\partial \Omega_t^+=:\Gamma_t,\\
&\Omega_t^-\cup \Gamma_t \cup \Omega_t^+=\mathbb{R}^N,\\
&\sup\{d(x,\Gamma_t);\ \ x\in \Omega_t^+\}=\sup\{d(x,\Gamma_t);\ \ x\in \Omega_t^-\}=+\infty
\end{aligned}
\right.
\end{eqnarray}
and
\begin{eqnarray}\label{eq1.4}
\left\{
\begin{aligned}
&\inf\Big\{\sup\big\{d(y,\Gamma_t);\ y\in \Omega_t^+,\ |y-x|\leq r\big\};\ \ t\in \mathbb{R},\ \ x\in \Gamma_t\Big\}\rightarrow +\infty\\
&\inf\Big\{\sup\big\{d(y,\Gamma_t);\ y\in \Omega_t^-,\ |y-x|\leq r\big\};\ \ t\in \mathbb{R},\ \ x\in \Gamma_t\Big\}\rightarrow +\infty
\end{aligned}
\right.
\text{ as}\ \ r\rightarrow +\infty.
\end{eqnarray}
From the condition~\eqref{eq1.3}, we notice that the interface $\Gamma_t$ is not empty for every $t\in \mathbb{R}$. As far as~\eqref{eq1.4} is concerned, it says that for any $M>0$, there is $r_M>0$ such that for any $t\in \mathbb{R}$ and $x\in \Gamma_t$, there are $y^{\pm}\in \mathbb{R}^N$ such that
\begin{eqnarray}\label{eq1.5}
y^{\pm}\in \Omega^{\pm}_t,\ \ |x-y^{\pm}|\leq r_M\ \ \text{and}\ \ d(y^{\pm},\Gamma_t)\geq M.
\end{eqnarray}
that is, $y^{\pm}\in \overline{B(x,r_M)}$ and $B(y^{\pm},M)\subset \Omega_t^{\pm}$, where $B(y,r)$ denotes the open Euclidean ball of center $y$ and radius $r>0$. Moreover, the sets $\Gamma_t$ are assumed to be made of a finite number of graphs: there is an integer $n\geq 1$ such that, for each $t\in \mathbb{R}$, there are $n$ open subsets $\omega_{i,t}\subset \mathbb{R}^{N-1}$(for $1\leq i\leq n$), $n$ continuous maps $\psi_{i,t}: \omega_{i,t}\rightarrow \mathbb{R}$ and $n$ rotations $R_{i,t}$ of $\mathbb{R}^N$, such that
\begin{equation}\label{eq1.6}
\Gamma_t \subset \bigcup_{1\leq i\leq n} R_{i,t}\left(\{x\in \mathbb{R}^N; \ \ x'\in \omega_{i,t},\ \ x_N=\psi_{i,t}(x')\}\right).
\end{equation}

\begin{definition}\label{TF}{\rm{\cite{BH1,BH2}}}
For problem~\eqref{eq1.1}, a transition front connecting $0$ and $1$ is a classical solution $u:\mathbb{R}\times\mathbb{R}^N \rightarrow (0,1)$ for which there exist some sets $(\Omega_t^{\pm})_{t\in \mathbb{R}}$ and $(\Gamma_t)_{t\in \mathbb{R}}$ satisfying~\eqref{eq1.3},~\eqref{eq1.4} and~\eqref{eq1.6}, and, for every $\varepsilon>0$, there exists $M_{\varepsilon}>0$ such that
\begin{eqnarray}\label{eq1.7}
\left\{\begin{aligned}
&\forall t\in \mathbb{R},\ \ \forall x\in \Omega_t^+, \ \ \left(d(x,\Gamma_t)\geq M_{\varepsilon}\right)\Rightarrow \left(u(t,x)\geq 1-\varepsilon\right)\!,\\
&\forall t\in \mathbb{R},\ \ \forall x\in \Omega_t^-, \ \ \left(d(x,\Gamma_t)\geq M_{\varepsilon}\right)\Rightarrow \left(u(t,x)\leq \varepsilon\right)\!.
\end{aligned}
\right.
\end{eqnarray}
Furthermore, $u$ is said to have a global mean speed $\gamma$ $(\geq 0)$ if
\begin{equation*}
\frac{d(\Gamma_t,\Gamma_s)}{|t-s|}\rightarrow \gamma \ \ \text{as}\ \ |t-s|\rightarrow +\infty.
\end{equation*}
\end{definition}

This definition has been shown in~\cite{BH1,BH2,H} to cover and unify all classical cases. Moreover, it was proved in~\cite{H} that, under some assumptions on $f$, any almost-planar transition front (in the sense that, for every $t\in\R$, $\Gamma_t$ is a hyperplane) connecting $0$ and $1$ is truly planar, and that any transition front connecting $0$ and $1$ has a global mean speed $\gamma$, which is equal to $|c_f|$. Non-standard transition fronts which are not invariant in any moving frame as time runs were also constructed in~\cite{H}. For other properties of bistable transition fronts, we refer to \cite{BH1,BH2,H}. There is now a large literature devoted to transition fronts in various homogeneous or heterogeneous settings or for other reaction terms, see e.g.~\cite{BHM,D,HR1,HR2,MNRR,MRS,N,NR1,NR2,NRRZ,NR4,S2,SS,Z1,Z2,Z3}.

Now, we present our results in this paper. Our first result is about the continuity of the speed $c_e$ and the profile $U_e$ with respect to $e\in\mathbb{S}^{N-1}$. Here, we can refer to \cite{AG} for the ignition type, in which the authors proved the continuity of the speed and the profile of the pulsating front with respect to the propagating direction.

\begin{theorem}\label{th1}
Assume that (A1), (A2) hold and $c_e>0$ for any $e\in\mathbb{S}^{N-1}$. Then, the speed $c_e$ and the profile $U_e$ are continuous with respect to $e\in\mathbb{S}^{N-1}$ under a normalization of the profile $U_e$, that is, $\int_{\R^+\times\mathbb{T}^N} U_e^2(\xi,y)dyd\xi=1$ for all $e\in\mathbb{S}^{N-1}$.
\end{theorem}

\begin{remark}
In Theorem \ref{th1}, the normalization of $U_e$ could be modified. In fact, we can normalize $U_e$ by the integral $\int_{\R^+\times\mathbb{T}^N} U_e^2(\xi,y)dyd\xi$ being any positive constant, or by $U_e(0,0)$ being any constant between $0$ and $1$ for all $e\in\mathbb{S}^{N-1}$.
\end{remark}

Normalize $U_e$ by $\int_{\R^+\times\mathbb{T}^N} U_e^2(\xi,y)dyd\xi=1$ for all $e\in\mathbb{S}^{N-1}$. For any $b\in\R^N\setminus\{0\}$, define
\be\label{Ub}
U_b=U_{\frac{b}{|b|}}\text{ and }c_b=c_{\frac{b}{|b|}}.
\ee
Then, $U_b$ and $c_b$ are well defined and continuous with respect to $b\in\R^N\setminus\{0\}$ by Theorem \ref{th1}.

\begin{theorem}\label{th2}
Normalize $U_e$ by $\int_{\R^+\times\mathbb{T}^N} U_e^2(\xi,y)dyd\xi=1$ and let $U_b$ and $c_b$ be defined in \eqref{Ub}. Then, $U_b$ and $c_b$ are doubly continuously Fr\'{e}chet differentiable at any $b\in\mathbb{R}^{N}\setminus\{0\}$.
\end{theorem}

Finally, we prove in this paper that the propagating rate of a transition front satisfies some estimates related to the speeds $c_e$ of pulsating fronts.

\begin{theorem}\label{th1.3}
Assume that (A1), (A2) hold and $c_e>0$ for any $e\in\mathbb{S}^{N-1}$. For any transition front $u(t,x)$ of \eqref{eq1.1}, it holds that
$$\inf_{e\in\mathbb{S}^{N-1}} c_e\le \liminf_{|t-s|\rightarrow +\infty} \frac{d(\Gamma_t,\Gamma_s)}{|t-s|}\le \limsup_{|t-s|\rightarrow +\infty} \frac{d(\Gamma_t,\Gamma_s)}{|t-s|}\le \sup_{e\in\mathbb{S}^{N-1}} c_e.$$
\end{theorem}

\begin{remark}\label{remark1.5}
By the continuity of $c_e$ from Theorem \ref{th1}, the $\inf$ and $\sup$ are actually $\min$ and $\max$. Moreover, since $c_e>0$ for any $e\in\mathbb{S}^{N-1}$, one has that $\inf_{e\in\mathbb{S}^{N-1}} c_e>0$ and~$\sup_{e\in\mathbb{S}^{N-1}} c_e<+\infty$.
\end{remark}

We point out that if (A1), (A2) do not hold, there may exist stationary pulsating fronts. In this situation, we will lose the continuity and differentiability of pulsating fronts in general. On the other hand, since $\inf_{e\in\mathbb{S}^{N-1}} c_e=0$ when there exist stationary fronts, the first inequality in Theorem \ref{th1.3} holds obviously. But we can not obtain the last inequality in Theorem \ref{th1.3} by our method since our proof is based on the continuity and differentiability of pulsating fronts.

We organize our paper as follows. In the next section, we investigate some properties of pulsating fronts. Especially we prove that the pulsating fronts $U_e$ and the speeds $c_e$ are continuous and Fr\'{e}chet differentiable with respect to the direction $e\in\mathbb{S}^{N-1}$, that is, we prove Theorem \ref{th1} and Theorem \ref{th2}. Section 3 is devoted to the proof of Theorem \ref{th1.3} by showing two key-lemmas in Section 3.1 and completing the proof in Section 3.2.


\section{Properties}
\noindent
In this section, we deduce some properties of pulsating fronts $U_e(x\cdot e-c_e t,x)$, which are well-known for planar fronts in homogeneous case. Especially, we prove the continuity and differentiability of $c_e$ and $U_e(\xi,y)$ with respect to the direction $e$, which obviously hold for homogeneous planar fronts since they are independent of the propagating direction.

\subsection{General properties}
\noindent
Since the properties in this section are proved for pulsating fronts in every direction $e$, we fix an arbitrary $e\in\mathbb{S}^{N-1}$ in this section. First, we prove that the pulsating fronts are approaching their limiting states $0$ and $1$ exponentially.

\begin{lemma}\label{ASY}
For any pulsating front $U_e(x\cdot e-c_e t,x)$ with $c_e\ge 0$, there exist $A_1$, $A_2\in \R$, $\mu_1>0$, $\mu_2>0$ ($\mu_1$, $\mu_2$ are independent of $e$), $C_1>0$, $C_2>0$ such that
\begin{align}
0<U_e(x\cdot e-c_e t,x)\le C_1 e^{-\mu_1(x\cdot e -c_e t)}\ \ \text{if}\ \ x\cdot e-c_e t\ge A_1,\label{eq:asymptotic1}\\
0<1-U_e(x\cdot e-c_e t,x)\le C_2 e^{\mu_2(x\cdot e-c_e t)}\ \ \text{if}\ \ x\cdot e-c_e t\le A_2.\label{eq:asymptotic2}
\end{align}
\end{lemma}

\begin{proof}
It is known by the strong maximum principle that $0<U_e(x\cdot e-c_e t,x)<1$ for all $(t,x)\in\R\times\R^N$. We only prove \eqref{eq:asymptotic1}, the proof being similar for \eqref{eq:asymptotic2}. We deal with it into two cases: $c_e=0$ and $c_e> 0$ (although assumption (A1) implies $c_e\neq 0$, we still deal with $c_e=0$ for completeness).

{\it Case 1: $c_e=0$.} In this case, the pulsating front $U_e(x\cdot e-c_e t,x)$ is a stationary front, that is, $U_e(x\cdot e-c_e t,x)=U_e(x\cdot e,x):=U(x)$. From Definition \ref{PF} of pulsating front, it satisfies
\begin{equation}\label{2.3}
-\Delta U-f(x,U)=0\,\text{ for $x\in\R^N$},
\end{equation}
and $\lim_{x\cdot e\rightarrow +\infty} U(x)=0$, $\lim_{x\cdot e\rightarrow -\infty} U(x)=1$. It means that there exists $A_1\in \R$ such that
\begin{equation}\label{2.4}
0<U(x)\le \sigma\, \text{ for all }\, x\cdot e\ge A_1.
\end{equation}
where $\sigma$ is defined in \eqref{F2}. From \eqref{F1}, \eqref{F2}, \eqref{2.3} and \eqref{2.4}, it follows that
\begin{equation}\label{2.5}
-\Delta U+\gamma U\le 0\,\text{ for all $x\cdot e\ge A_1$},
\end{equation}
where $\gamma>0$ is also given in \eqref{F2}.

Define $\omega(x)=\sigma e^{-\mu_1 (x\cdot e-A_1)}$ where $\mu_1$ is a positive constant to be chosen. The function $\omega$ satisfies
\begin{equation*}
-\Delta \omega+\gamma\omega=(-\mu_1^2+\gamma)\sigma e^{-\mu_1 (x\cdot e-A_1)}\,\text{ for $x\in\R^N$}.
\end{equation*}
Take $\mu_1=\sqrt{\gamma}$ so that $-\mu_1^2+\gamma=0$ which also means $-\Delta \omega+\gamma \omega=0$ for $x\in\R^N$. Since $U(x)\rightarrow 0$ as $x\cdot e\rightarrow +\infty$ and $\omega(x)\ge U(x)$ for all $x\cdot e=A_1$ from \eqref{2.4}, it follows from \eqref{2.5} and the elliptic weak maximum principle, that
$$U(x)\le \sigma e^{-\mu_1 (x\cdot e-A_1)}\, \text{ for }x\cdot e\ge A_1.$$

{\it Case 2: $c_e> 0$.} In this case, we consider the pulsating front $v(t,x):=U_e(x\cdot e-c_e t,x)$ which satisfies \eqref{eq1.1} with limiting conditions $\lim_{x\cdot e-c_e t\rightarrow \pm\infty} v(t,x)=0,\ 1$. It means that there exists $A_1\in\R$ such that
\begin{equation}\label{2.8}
0<v(t,x)\le \sigma\,\text{ for all $x\cdot e-c_e t\ge A_1$}.
\end{equation}
From \eqref{F2} and \eqref{2.8}, it follows that
\begin{equation}\label{eq:v}
v_t-\Delta v+\gamma v\le 0\,\text{ for all $x\cdot e-c_e t\ge A_1$}.
\end{equation}

Define $\omega(t,x)=\sigma e^{-\mu_1 (x\cdot e -c_e t-A_1)}$ for $\mu_1=\sqrt{\gamma}>0$ such that $\mu_1 c_e -\mu_1^2+\gamma=\mu_1 c_e\ge 0$. Then $\omega(t,x)$ satisfies
\begin{equation}\label{eq:omega}
\omega_t-\Delta \omega+\gamma \omega\ge 0\,\text{ for all $(t,x)\in\R\times\R^N$}.
\end{equation}
On the other hand,
$$\delta\ge U_e(A_1,x)\,\text{ for all $x\in \T^N$},$$
that is, $\omega(t,x)\ge v(t,x)$ for all $x\cdot e-c_e t=A_1$. Let
$$\varepsilon^*=\inf\{\varepsilon> 0;\ v(t,x)-\varepsilon\le \omega(t,x) \,\text{ for all }x\cdot e-c_e t\ge A_1\}$$
which is well-defined from \eqref{2.8} and $\omega(t,x)>0$. We only need to show $\varepsilon^*=0$.

Assume by contradiction that $\varepsilon^*>0$. There exist then a sequence $(\varepsilon_{n})_{n\in \N}$ of positive real numbers and a sequence of points $(t_n,x_n)_{n\in \N}$ satisfying $x_n\cdot e-c_e t_n\ge A_1$ such that
\begin{equation}\label{v}
\varepsilon_n\rightarrow \varepsilon^*\,\text{ as $n\rightarrow +\infty$ and }\, v(t_n,x_n)-\varepsilon_n> \omega(t_n,x_n)\,\text{ for all $n\in\N$}.
\end{equation}

We claim that $x_n\cdot e-c_e t_n -A_1\ge 0$ are upper-bounded uniformly in $n\in\N$. Otherwise, $v(t_n,x_n)\rightarrow 0$ and $\omega(t_n,x_n)\rightarrow 0$ which means $-\varepsilon^*\ge 0$ from \eqref{v} and then contradicts $\varepsilon^*>0$. Therefore, $\xi_n:=x_n\cdot e-c_e t_n$ are bounded and $v(t_n,x_n)=U(\xi_n,x_n)$, $\omega(t_n,x_n)=e^{-\mu_1 \xi_n}$. Since $U(\xi,y)$ is periodic in $y$, there is then $(\xi^*,x^*)\in\R\times\R^N$ or say, $(t^*,x^*)\in\R\times\R^N$ such that $x^*\cdot e-c_e t^* > A_1$ and  $v(t^*,x^*)-\varepsilon^*= \omega(t^*,x^*)$. Define $z=\omega-v$. From \eqref{eq:v} and \eqref{eq:omega}, it follows that $z_t-\Delta z+\gamma z\ge 0$ for all $x\cdot e-c_e t\ge A_1$. But $z$ reaches a minimum at the point $(t^*,x^*)$ with $x^*\cdot e-c_e t^*>A_1$ and $z(t^*,x^*)=-\varepsilon^*<0$. Thus, $-\gamma\varepsilon^*\ge 0$, which is a contradiction. Therefore, $\varepsilon^*=0$, that is, \eqref{eq:asymptotic1} holds. This completes the proof.
\end{proof}
\vskip 0.3cm

Although the following lemma is elementary, we state it for completeness.

\begin{lemma}\label{lemma2.2}
For any pulsating front $U_e(x\cdot e-c_e t,x)$ with $c_e\neq 0$, the speed $c_e$ has the sign of $\int_{\T^N\times[0,1]} f(x,u)dxdu$.
\end{lemma}

\begin{proof}
Notice that $u(t,x)=U_e(x\cdot e-c_e t,x)$ is a classical solution of \eqref{eq1.1} and $v=u_t$ is a classical solution of $v_t=\Delta v+f_u(x,u)v$. Then, by Lemma \ref{ASY} and standard parabolic estimates, all functions $\partial_{\xi} U_e$, $\partial_{y_i} U_e$, $\partial_{\xi\xi} U_e$, $\partial_{y_i\xi} U_e$, and $\partial_{y_iy_j} U_e$ for $i$, $j=1,\cdots,N$, converge to $0$ exponentially as $\xi\rightarrow\pm \infty$. Integrating \eqref{Ue} in $\R\times\T^N$ by parts against $\partial_{\xi} U_e$, one has that
$$c_e\int_{\R\times\T^N} |\partial_{\xi} U_e|^2 dyd\xi=\int_{\T^N\times[0,1]} f(y,u)dydu.$$
Thus, $c_e$ has the sign of $\int_{\T^N\times[0,1]} f(x,u)dxdu$.
\end{proof}
\vskip 0.3cm

In the next lemma, we show that every pulsating front with nonzero speed is strictly monotone in time.

\begin{lemma}\label{lemma2.1}
Any pulsating front $U_e(x\cdot e-c_e t,x)$ with $c_e\neq 0$ is monotone in $t$.
\end{lemma}

\begin{proof}
By Definition \ref{TF} of transition fronts, one can notice that, any pulsating front $U_e(x\cdot e-c_e t,x)$ is a transition front with $(\Gamma_t)_{t\in\mathbb{R}}:=(c_e t e)_{t\in\mathbb{R}}$, $(\Omega^+_t)_{t\in\mathbb{R}}:=(\{x|x\cdot e<c_e t\})_{t\in\mathbb{R}}$,  $(\Omega^-_t)_{t\in\mathbb{R}}:=(\{x|x\cdot e>c_e t\})_{t\in\mathbb{R}}$. Moreover, from \eqref{F1}, \eqref{F2} and the regularity of $f$, there exists a positive constant $\hat{\sigma}$ such that the function $f(x,s)$ is nonincreasing in $[0,\hat{\sigma}]$ and in $[1-\hat{\sigma},1]$. Therefore, from \cite[Definition 1.4]{BH2}, $U_e(x\cdot e-c_e t,x)$ is an invasion of $0$ by $1$ when $c_e>0$. Then, by \cite[Theorem 1.11]{BH2} with its followed discussion, it implies that $U_e(x\cdot e-c_e t,x)$ is increasing in $t$. Similarly when $c_e<0$, the pulsating front is an invasion of $1$ by $0$, and whence it is decreasing in $t$. From the strong maximum principle applied to $u_t$, this also implies that $\partial_{\xi} U_e(\xi,y)<0$ for all $(\xi,y)\in\R\times\R^N$ which completes the proof.
\end{proof}
\vskip 0.3cm

\begin{lemma}\label{lemma2.4}
For every direction $e\in\mathbb{S}^{N-1}$, the speed of pulsating fronts for \eqref{eq1.1} with non-zero speed is unique in the sense that if $U_e(x\cdot e-c_e t,x)$ and $\tilde{U}_e(x\cdot e-\tilde{c}_e t,x)$ are two pulsating fronts with $c_e\neq 0$, $\tilde{c}_e\neq 0$, then $c_e=\tilde{c}_e$. Furthermore, the pulsating front is unique up to shifts in t, that is, there is $\tau\in\R$ such that $\tilde{U}_e(x\cdot e-\tilde{c}_e t,x)=U_e(x\cdot e-c_e t+\tau,x)$.
\end{lemma}

\begin{proof}
Under the assumptions of Lemma \ref{lemma2.4}, Lemma \ref{lemma2.2} implies that $c_e$ and $\tilde{c}_e$ have that same sign. If follows then from \cite[Thoerem 1.12 and 1.14]{BH2} that $c_e=\tilde{c}_e$ and that the fronts are unique up to shifts in time.
\end{proof}

\subsection{Continuity}
\noindent
This section is devoted to proving the continuity of $(U_e,c_e)$ with respect to the direction $e$.

Following the proof of \cite[Theorem 1.4]{DHZ1}, we can get a uniform bound of the speeds of pulsating fronts for any direction.

\begin{lemma}\label{lemma2.3}
There is a positive constant $C$ depending only on the function $f$ such that
$$\sup_{e\in\mathbb{S}^{N-1}} |c_e|\le C.$$
\end{lemma}

\begin{remark}
The strategy for the proof of Lemma \ref{lemma2.3} as in \cite{DHZ1}, is to construct supersolutions and subsolutions of \eqref{eq1.1} as
$$\overline{u}(t,x)=\min\left(e^{-(x\cdot e-Ct)}+\frac{\sigma}{2}e^{-\gamma t}, 1\right),\text{ for $t\ge 0$ and $x\in\R^N$},$$
and
$$\underline{u}(t,x)=\max\left(1-e^{(x\cdot e+Ct)}-\frac{\sigma}{2}e^{-\gamma t}, 0\right),\text{ for $t\ge 0$ and $x\in\R^N$},$$
where $\sigma$ and $\gamma$ are given in \eqref{F2} and $C>0$ is a sufficiently large constant independent of the direction $e$.
\end{remark}

We now prove the continuity of $(U_{e},c_{e})$, that is, Theorem \ref{th1}.
\vskip 0.3cm

\begin{proof}[Proof of Theorem \ref{th1}]
{\it Step 1: proof of $\inf_{e\in\mathbb{S}^{N-1}} c_e>0$.} We first show that $\inf_{e\in\mathbb{S}^{N-1}} c_e>0$. Assume by contradiction that there is a sequence $(e_n)_{n\in\mathbb{N}}\subset \mathbb{S}^{N-1}$ such that $c_{e_n}\rightarrow 0$ as $n\rightarrow +\infty$. We assume that there is $e_0\in\mathbb{S}^{N-1}$ such that $e_n\rightarrow e_0$ as $n\rightarrow +\infty$, even if it means to extract a subsequence. For every direction $e\in\mathbb{S}^{N-1}$, we normalize $U_e$ by
\be\label{2.10}
U_e(0,0)=1-\delta',
\ee
where $\delta'>0$ will be defined later. Let $u_n(t,x)=U_{e_n}(x\cdot e_n-c_{e_n}t,x)$. Since $\partial_{\xi} U_e$ is negative for all $e\in\mathbb{S}^{N-1}$ and $U_{e_n}(\xi,y)$ is periodic in $y$, it follows that
\be\label{eq:2.11}
u_n(1,x)\ge 1-\delta', \text{ for $x\in\mathbb{Z}^N$ such that $x\cdot e_n-c_{e_n}\le 0$}.
\ee
By standard parabolic estimates, $u_n$ converges locally uniformly, up to a subsequence, to a solution $u_{\infty}$ of \eqref{eq1.1}. By $(u_n)_t>0$, one has that $(u_{\infty})_t\ge 0$. Furthermore, by \eqref{eq:2.11}, $e_n\rightarrow e_0$ and $c_{e_n}\rightarrow 0$ as $n\rightarrow +\infty$, it follows that
\be\label{eq:2.12}
u_{\infty}(1,x)\ge 1-\delta', \text{ for $x\in\mathbb{Z}^N$ such that $x\cdot e_0\le 0$}.
\ee
Let $\delta'>0$ be chosen less than $1$ and whence $u_{\infty}(1,x)\ge 1-\delta'>0$ for $x\in\mathbb{Z}^N$ such that $x\cdot e_0\le 0$ and $u_{\infty}(0,0)=1-\delta'<1$. By the strong maximum principle, it follows that $0<u_{\infty}(t,x)<1$ for all $(t,x)\in\R\times\R^N$.

Let $\delta>0$ be such that
$$\delta<\min(\gamma,\sigma),$$
where $\gamma$ and $\sigma$ are defined in \eqref{F2}. Since $\lim_{\xi\rightarrow -\infty} U_{e_0}(\xi,y)=1$ and $\lim_{\xi\rightarrow +\infty} U_{e_0}(\xi,y)=0$, there is $C>0$ such that
\be\label{2.11}
U_{e_0}(\xi,y)\ge 1-\delta, \text{ for $\xi\le -C$ and } U_{e_0}(\xi,y)\le \delta, \text{ for $\xi\ge C$}.
\ee
Since $\partial_{\xi} U_{e_0}(\xi,y)$ is negative and continuous in $\R\times\mathbb{T}^N$, there is $k>0$ such that $-\partial_{\xi} U_{e_0}\ge k$ for all $(\xi,y)\in [-C,C]\times\mathbb{T}^N$. Let $\omega>0$ such that
$$\omega k\ge L+\delta,$$
where $L=\max_{(u,x)\in [0,1]\times\mathbb{T}^N} |f_u(u,x)|$. From \eqref{eq:2.12}, the Harnack inequality and $1$ is a solution of \eqref{eq1.1}, one can choose $\delta'$ small enough such that
\be\label{eq2.11}
u_{\infty}(0,x)\ge 1-\delta, \text{ for $x\in\R^N$ such that $x\cdot e_0 \le 0$}.
\ee

Then, for any $(t,x)\in\R\times\R^N$, we set
\be\label{eq:2.15}
\underline{u}(t,x)=\max\left(U_{e_0}(x\cdot e_0 -c_{e_0} t -\omega e^{-\delta t}+\omega + C,x)-\delta e^{-\delta t},0\right).
\ee
Let us check that $\underline{u}$ is a subsolution for the problem satisfied by $u_{\infty}(t,x)$, for $t\ge 0$ and $x\in\R^N$. First, at the time $0$, it follows from \eqref{eq2.11} that
$$u_{\infty}(0,x)\ge 1-\delta\ge \underline{u}(0,x),\text{ for all $x\in\R^N$ such that $x\cdot e_0\le 0$}.$$
On the other hand, from \eqref{2.11} and the fact that $u_{\infty}\ge 0$, it follows that for all $x\in\R^N$ such that $x\cdot e_0\ge 0$,
$$\underline{u}(0,x)=\max\left(U_{e_0}(x\cdot e_0 + C,x)-\delta ,0\right)\le\max(0,0)=0\le u_{\infty}(0,x).$$
Thus,
$$u_{\infty}(0,x)\ge \underline{u}(0,x), \text{ for all $x\in\R^N$}.$$
Inspired by \cite{FM} and \cite{H}, it is easy to check that
$$L \underline{u}=\underline{u}_t -\Delta \underline{u}-f(\underline{u})\le 0$$
for all $t\ge 0$ and $x\in\R^N$ such that $\underline{u}(t,x)>0$. By the comparison principle, one gets that
$$u_{\infty}(t,x)\ge \underline{u}(t,x), \text{ for $t\ge 0$ and $x\in\R^N$}.$$
Since $c_{e_0}>0$ and $\lim_{\xi\rightarrow -\infty} U_{e_0}(\xi,y)=1$, one infers that $u_{\infty}(t,x)$ converges locally uniformly to~$1$ as $t\rightarrow +\infty$.

Fix $l\in\Z^N$ such that $l\cdot e_0>0$. Since $e_n\rightarrow e_0$ and $c_{e_n}\rightarrow 0$ as $n\rightarrow +\infty$, one has that $l\cdot e_n>0$ for $n$ large enough, and $l\cdot e_n / c_{e_n}\to +\infty$ as $n\to +\infty$. Then, for any $s\in\R$, it follows from the definition of pulsating fronts and $(u_n)_t>0$ that
$$u_n(s,0)\le u_n(\frac{l\cdot e_n}{c_{e_n}},0)=u_n(0,-l),$$
for $n$ large enough. Passing to the limit as $n\rightarrow +\infty$, it follows that
$$u_{\infty}(s,0)\le u_{\infty}(0,-l)<1,$$
for all $s\ge 0$. This contradicts the locally uniform convergence of $u_{\infty}(t,x)$ to $1$ as $t\rightarrow +\infty$. Thus, we get that $\inf_{e\in\mathbb{S}^{N-1}} c_e>0$.

{\it Step 2: continuity of $c_e$.} Take any $e_0\in\mathbb{S}^{N-1}$ and any sequence $(e_n)_{n\in\mathbb{N}}\subset \mathbb{S}^{N-1}$ such that $e_n\rightarrow e_0$ as $n\rightarrow +\infty$. Then, by Lemma \ref{lemma2.3} and Step 1, there is $c> 0$ and a subsequence $c_{e_{n_k}}$ such that $c_{e_{n_k}}\rightarrow c$ as $n_k\rightarrow +\infty$. For all direction $e\in\mathbb{S}^{N-1}$, we still take the normalization \eqref{2.10}. By standard parabolic estimates applied to $u(t,x)=U_e(x\cdot e-c_e t,x)$ for all $e\in\mathbb{S}^{N-1}$, one gets that $U_{e}$ and its derivatives are uniformly bounded in $\R\times\mathbb{T}^{N}$ and uniformly for $e\in\mathbb{S}^{N-1}$. Then, the sequence $U_{e_{n_k}}$ converges locally uniformly along with its derivatives up to the second order, up to a subsequence, to a function $U_{\infty}$ and $U_{\infty}$ satisfies
$$c \partial_{\xi} U_{\infty} +\partial_{\xi\xi} U_{\infty} +2\nabla_y \partial_{\xi} U_{\infty}\cdot e_0+\Delta_y U_{\infty} +f(y,U_{\infty})=0,\,\text{ for all $(\xi,y)\in\R\times\T^N$},
$$
and $U_{\infty}(0,0)=1-\delta'$.
That also implies that if let $v_n(t,x)=U_{e_{n_k}}(x\cdot e_{n_k}-c_{e_{n_k}} t,x)$, one has that $v_n(t,x)\rightarrow v_{\infty}(t,x)=U_{\infty}(x\cdot e_0- c t,x)$ locally uniformly in $\R\times\R^N$ and $v_{\infty}(t,x)$ satisfies~\eqref{eq1.1}. Moreover, since $U_e(\xi,y)$ is periodic in $y$ and $\partial_{\xi} U_e(\xi,y)<0$ for all $e\in\mathbb{S}^{N-1}$, one has that $U_{\infty}(\xi,y)$ is periodic in $y$ and $\partial_{\xi} U_{\infty}(\xi,y)\le 0$.

We borrow the parameters $\delta$, $\omega$, $k$ from Step 1. By the normalization \eqref{2.10} and $U_{\infty}(\xi,y)$ is periodic in $y$ and nonincreasing in $\xi$, one gets that $v_{\infty}(t+1,x)=U_{\infty}(x\cdot e_0-c(t+1),x)\ge 1-\delta'$ for all $t\in\R$ and $x\in\mathbb{Z}^N$ such that $x\cdot e_0-c(t+1)\le 0$. From the Harnack inequality and $1$ is a solution of \eqref{eq1.1}, one can choose $\delta'$ small enough such that
$$v_{\infty}(t,x)=U_{\infty}(x\cdot e_0-ct,x)\ge 1-\delta, \text{ for all $x\cdot e_0-c t\le 0$}.$$
Then, one can prove as in Step~1 that $\underline{u}(t,x)$ defined in \eqref{eq:2.15} is a subsolution of the problem satisfied by $v_{\infty}(t,x)$, for $t\ge 0$ and $x\in\R^N$.

By the comparison principle, one gets that
$$v_{\infty}(t,x)=U_{\infty}(x\cdot e_0 -ct,x)\ge \underline{u}(t,x), \text{ for $t\ge 0$ and $x\in\R^N$}.$$
This implies that $c\ge c_{e_0}$. In fact, if $c<c_{e_0}$, one has that for any $(t,x)\in (0,+\infty)\times\R^N$ such that $x\cdot e_0=ct$, $x\cdot e_0 -c_{e_0} t -\omega e^{-\delta t}+\omega + C=-(c_{e_0}-c) t -\omega e^{-\delta t}+\omega + C\rightarrow -\infty$ as $t\rightarrow +\infty$. Since $\lim_{\xi\rightarrow -\infty} U_{e_0}(\xi,y)=1$ and $\lim_{t\rightarrow +\infty} e^{-\delta t}=0$, there exists $T>0$ large enough such that for any $x\in\R^N$ such that $x\cdot e_0=cT$,
\begin{align}
v_{\infty}(T,x)\ge\underline{u}(T,x)=&\max\left(U_{e_0}(x\cdot e_0 -c_{e_0} T -\omega e^{-\delta T}+\omega + C,x)-\delta e^{-\delta T},0\right)
\ge 1-\frac{\delta'}{2}.\label{2.13}
\end{align}
However, for any $x\in\mathbb{Z}^{N}$ such that $x\cdot e_0=cT$, it follows that $v_{\infty}(T,x)=U_{\infty}(0,x)=U_{\infty}(0,0)=1-\delta'$ since $U_{\infty}(\xi,y)$ is periodic in $y$ which is a contradiction with \eqref{2.13}.

Now we prove $c\le c_{e_0}$. Take $z_{n_k}$ such that $U_{e_{n_k}}(z_{n_k},0)=\delta'$. Then, from the analysis of the head of this step, one has that $v'_{n_k}(t,x)=U_{e_{n_k}}(x\cdot e_{n_k}-c_{n_k} t + z_{n_k},x)$ converge locally uniformly, up to a subsequence, to a solution $v'_{\infty}(t,x)=U'_{\infty}(x\cdot e_0-c t,x)$ of \eqref{eq1.1} where $U'_{\infty}(0,0)=\delta'$, $\partial_{\xi} U'_{\infty}\le 0$ and $U'_{\infty}(\xi,y)$ is periodic in $y$. Then, one can construct supersolutions for the problem satisfied by $v'_{\infty}(t,x)$ as
$$\overline{u}(t,x)=\min\left(U_{e_0}(x\cdot e_0 -c_{e_0} t +\omega e^{-\delta t}-\omega - C,x)+\delta e^{-\delta t},1\right),$$
for $t\ge 0$ and $x\in\R^N$. Similar to the arguments as above, one infers that $c\le c_{e_0}$.

Then, one can conclude that $c=c_{e_0}$. By the uniqueness of $c_{e_0}$ in the direction $e_0$ and $e_0$ is arbitrary taken, it implies that $c_e$ is continuous with respect to $e\in\mathbb{S}^{N-1}$.

{\it Step 3: continuity of $U_e$ under a normalization.} We now prove the continuity of $U_e$ under the normalization
\be\label{normalization}
\int_{\R^+\times\mathbb{T}^N} U_e^2(\xi,y)dyd\xi=1.
\ee
Take any $e_0\in\mathbb{S}^{N-1}$ and any sequence $(e_n)_{n\in\mathbb{N}}\subset \mathbb{S}^{N-1}$ such that $e_n\rightarrow e_0$ as $n\rightarrow +\infty$. Remember that $c_{e_n}\rightarrow c_{e_0}>0$ from the continuity of $c_e$. Let $\xi_n$ such that $\sup_{y\in\R^N} U_{e_n}(\xi_n,y)= \sigma$, where~$\sigma$ is defined in \eqref{F2} (remember also that $\sigma<\theta_x$ for all $x\in\R^N$).
Then, by standard parabolic estimates applied to the fronts $(t,x)\mapsto U_{e_n}(x\cdot e_n-c_{e_n} t,x)$ and since $c_{e_n}\rightarrow c_{e_0}>0$, the sequence $U_{e_{n}}(\cdot+\xi_n,\cdot)$ converges locally uniformly along with its derivatives up to the second order, up to a subsequence, to a function $U_{\infty}$ and $U_{\infty}$ satisfies
$$c_{e_0} \partial_{\xi} U_{\infty} +\partial_{\xi\xi} U_{\infty} +2\nabla_y \partial_{\xi} U_{\infty}\cdot e_0+\Delta_y U_{\infty} +f(y,U_{\infty})=0,\,\text{ for all $(\xi,y)\in\R\times\T^N$},
$$
and $\sup_{y\in\R^N} U_{\infty}(0,y)=\sigma$.
Since $U_e(\xi,y)$ is periodic in $y$ and $\partial_{\xi} U_e(\xi,y)<0$ for all $e\in\mathbb{S}^{N-1}$, one has that $U_{\infty}(\xi,y)$ is periodic in $y$ and $\partial_{\xi} U_{\infty}(\xi,y)\le 0$.
Thus, there are periodic functions $p^+(y)$ and $p^-(y)$ such that $\lim_{\xi\rightarrow -\infty} U_{\infty}(\xi,y)=p^+(y)$ and $\lim_{\xi\rightarrow +\infty} U_{\infty}(\xi,y)=p^-(y)$. Moreover, by standard parabolic estimates applied to $u_{\infty}(t,x)=U_{\infty}(x-c_{e_0}t,x)$, we get that $p^{\pm}(y)$ are $C^2(\R^N)$ periodic stationary solutions of \eqref{eq1.1}. From $\sup_{y\in\R^N} U_{\infty}(0,y)=\sigma$, it follows that $p^-(y)\le \sigma$. Then, by the strong maximum principle, $p^-(y)\equiv 0$. If $p^+(y)\equiv 1$, it implies that $u_{\infty}(t,x)=U_{\infty}(x\cdot e_0-c_{e_0}t,x)$ is a pulsating front connecting $0$ and $1$. Then, by Lemma \ref{lemma2.4}, one has that $U_{\infty}$ equals to $U_{e_0}$ up to shifts.

Assume by contradiction that $p^+(y)\not\equiv 1$. From the strong maximum principle, $p^+(y)<1$. Set $r=\sup_{x\in\mathbb{T}^N} p^+(y)<1$. Then, $U_{\infty}(\xi,y)\le r<1$ for all $(\xi,y)\in\R\times\mathbb{T}^N$ since $\partial_{\xi} U_{\infty}(\xi,y)\le 0$.

Let $u(t,x)=U_{e_0}(x\cdot e_0-c_{e_0} t,x)$ and $u_{\infty}(t,x)=U_{\infty}(x\cdot e_0-c_{e_0}t,x)$. Notice that $u_{\infty}(t,x)>0$ from the maximum principle, since $\sup_{y\in\R^N} U_{\infty}(0,y)=\sigma>0$ and $u_{\infty}\ge 0$. Let $\delta'>0$ such that $f(x,\cdot)$ is nonincreasing in $(-\infty,\delta']$. Since $U_{\infty}(\xi,y)$ is nonincreasing in $\xi$ and $\lim_{\xi\rightarrow +\infty} U_{\infty}(\xi,y)=p^-(y)=0$, there is a constant $A$ such that
$$u_{\infty}(t,x)=U_{\infty}(x\cdot e_0 -c_{e_0}t,x)\le \delta',\text{ for all $(t,x)\in\R\times\R^N$ such that $x\cdot e_0-c_{e_0}t\ge A$}.$$
Since $\lim_{\xi\rightarrow -\infty} U_{e_0}(\xi,y)=1$, there is $\tau>0$ such that
$$u(t+\tau,x)=U_{e_0}(x\cdot e_0-c_{e_0} t-c_{e_0} \tau,x)\ge r,\text{for all $(t,x)\in\R\times\R^N$ such that $x\cdot e_0-c_{e_0}t\le A$}.$$
Then, $u_{\infty}(t,x)\le u(t+\tau,x)$ for all $(t,x)\in\R\times\R^N$ such that $x\cdot e_0-c_{e_0}t\le A$ since $u_{\infty}(t,x)=U_{\infty}(x\cdot e_0-c_{e_0}t,x)\le r$. Define
$$\omega^-=\{(t,x)\in\R\times\R^N;\ x\cdot e_0-c_{e_0}t\ge A\}.$$
One can follow the proof of \cite[Lemma 4.2]{BH2} to get that $u_{\infty}(t,x)\le u(t+\tau,x)$ for $(t,x)\in\omega^-$. Then, $u_{\infty}(t,x)\le u(t+\tau,x)$ for all $(t,x)\in \R\times\R^N$.

Define
$$\tau^*=\inf\{\tau'\in\R;\ u_{\infty}(t,x)\le u(t+\tau',x)\text{ for all $(t,x)\in\R\times\R^N$}\}.$$
Observe that $\tau^*\in\R$ is well defined, since $u(t+\tau',x)\rightarrow 0$ as $\tau'\rightarrow -\infty$ for every $(t,x)\in\R\times\R^N$, while $u_{\infty}(t,x)>0$. Since $u(t,x)=U_{e_0}(x\cdot e_0-c_{e_0} t,x)$ and $\lim_{\xi\rightarrow -\infty} U_{e_0}(\xi,y)=1$, there are some $B>0$ such that $u(t+\tau^*,x)\ge (1+r)/2$ for any $(t,x)\in\R\times\R^N$ such that $x\cdot e-c_{e_0}t\le -B$. Note that $u_{\infty}(t,x)\le r<(1+r)/2<1$. Then, assume that $\inf_{-B\le x\cdot e_0-c_{e_0}t\le A} (u(t+\tau^*,x)-u_{\infty}(t,x))>0$ and $u_{\infty}(t,x)<u(t+\tau^*,x)$ for all $(t,x)\in\R\times\R^N$ such that $-B\le x\cdot e_0-c_{e_0}t\le A$. Then, there is $\eta_0>0$ such that for $\eta\in(0,\eta_0)$,
$$u_{\infty}(t,x)\le u(t+\tau^*-\eta,x),\text{ for all $(t,x)\in\R\times\R^N$ such that $-B\le x\cdot e_0-c_{e_0}t\le A$}.$$
Then, followed again the proof of \cite[Lemma 4.2]{BH2}, one has that $u_{\infty}(t,x)\le u(t+\tau^*-\eta,x)$ for $(t,x)\in\omega^-$ and also for all $x\cdot e_0-c_{e_0}t\le -B$, from the choice of $B$. Thus, $u_{\infty}(t,x)\le u(t+\tau^*-\eta,x)$ for all $(t,x)\in\R\times\R^N$ which contradicts the definition of $\tau^*$. Therefore,
$$\inf\{u(t+\tau^*,x)-u_{\infty}(t,x);\ -B\le x\cdot e_0-c_{e_0}t\le A\}=0.$$
Then, there is a sequence $(t_n,x_n)$ such that $-B\le x_n\cdot e_0-c_{e_0}t_n\le A$ and $u_{\infty}(t_n,x_n)=u(t_n+\tau^*,x_n)$. By periodicity of $U_{e_0}(\xi,y)$ and $U_{\infty}(\xi,y)$ with respect to $y$, one can assume without loss of generality that the sequence $(x_n)_{n\in\mathbb{N}}$ is bounded and that there is $(t^*,x^*)\in \R\times\R^N$ such that $x_n\rightarrow x^*$ and $t_n\rightarrow t^*$ as $n\rightarrow +\infty$. Therefore, $u_{\infty}(t^*,x^*)=u(t^*+\tau^*,x^*)$ and $u_{\infty}(\cdot,\cdot)\le u(\cdot+\tau^*,\cdot)$ in $\R\times\R^N$. The strong maximum principle implies that $u_{\infty}(\cdot,\cdot)\equiv u(\cdot+\tau^*,\cdot)$ in $\R\times\R^N$, which is a contradiction, since $u_{\infty}\le r$ in $\R\times\R^N$. Thus, $p^+(y)\equiv 1$ and whence $U_{\infty}$ equals to $U_{e_0}$ up to shifts.

Now we show that the sequence of shifts $\xi_n$ defined by $\sup_{y\in\R^N} U_{e_n}(\xi_n,y)= \sigma$ is bounded. Assume first by contradiction that, up to extraction of a subsequence, $\xi_n \rightarrow -\infty$ as $n\rightarrow +\infty$. Since $\sup_{y\in\R^N} U_{e_n}(\xi_n,y)=\sigma$ and $\partial_{\xi} U_{e_n}(\xi,y)<0$, one has that $U_{e_n}(\xi_n+\xi,y)\le \sigma$ for $\xi\ge 0$ and $y\in\R^N$. Followed by the proof of Lemma \ref{lemma2.1}, one gets that $U_{e_n}(\xi_n+\xi,y)\le \sigma e^{-\mu_1 \xi}$ for $\xi\ge 0$ and $y\in\R^N$, where $\mu_1$ is independent of $e_n$. Then, the normalization \eqref{normalization} implies that
$$1=\int_{\R^+\times\mathbb{T}^N} U_{e_n}^2(\xi,y)dyd\xi=\int_{(-\xi_n,+\infty)\times\mathbb{T}^N} U_{e_n}^2(\xi_n+\xi,y)dyd\xi\le \int_{(-\xi_n,+\infty)\times\mathbb{T}^N}\sigma^2 e^{-2\mu_1 \xi} d\xi\rightarrow 0,$$
as $\xi_n\rightarrow -\infty$, which is a contradiction. Then, consider that $\xi_n\rightarrow +\infty$ as $n\rightarrow +\infty$. By the normalization \eqref{normalization}, one has that $\int_{(-\xi_n,+\infty)\times\mathbb{T}^N} U_{e_n}^2(\xi_n+\xi,y)dyd\xi=1$.
Since, from the previous paragraph, $U_{e_n}(\xi_n+\xi,y)\rightarrow U_{e_0}(\xi+\xi_0,y)$ locally uniformly in $\R\times\mathbb{R}^N$ for some $\xi_0\in\R$, we get that $$\int_{[-K,K]\times\mathbb{T}^N} U_{e_n}^2(\xi_n+\xi,y)dyd\xi\rightarrow\int_{[-K,K]\times\mathbb{T}^N} U_{e_0}^2(\xi+\xi_0,y)dyd\xi$$
for any $K>0$ as $n\rightarrow +\infty$. Since $\xi_n\rightarrow +\infty$ as $n\rightarrow +\infty$, one has that for all $K>0$,
\begin{align*}
\int_{[-K,K]\times\mathbb{T}^N} U_{e_0}^2(\xi+\xi_0,y)dyd\xi\le&\lim_{n\rightarrow +\infty} \int_{[-K,K]\times\mathbb{T}^N} U_{e_n}^2(\xi_n+\xi,y)dyd\xi\\
\le& \lim_{n\rightarrow +\infty} \int_{(-\xi_n,+\infty)\times\mathbb{T}^N} U_{e_n}^2(\xi_n+\xi,y)dyd\xi=1.
\end{align*}
The limit as $K\rightarrow +\infty$ leads to a contradiction, since $U_{e_0}(\xi,y)\rightarrow 1$ as $\xi\rightarrow -\infty$. Thus, $\xi_n$ is bounded and up to extraction of a subsequence, $U_{e_n}(\xi,y)\rightarrow U_{e_0}(\xi+\xi_0,y)$ locally uniformly in $\R\times\mathbb{R}^N$ for some $\xi_0\in\R$ as $n\rightarrow +\infty$.

Then, we prove that the convergence $U_{e_n}(\xi,y)\rightarrow U_{e_0}(\xi+\xi_0,y)$ is in fact uniform in $\R\times\mathbb{R}^N$. Note that the uniformity with respect to the second variable $y$ immediately follows from the periodicity. Furthermore, for a given $\varepsilon>0$, let $K>0$ be such that
$$0\le U_{e_0}(\xi+\xi_0,y)\le \frac{\varepsilon}{2} \text{ for $\xi\ge K$, $y\in\R^N$ and } 1-\frac{\varepsilon}{2}\le U_{e_0}(\xi+\xi_0,y)\le 1 \text{ for $\xi\le -K$, $y\in\R^N$}.$$
Then, for $n$ large enough, one has that
$$\|U_{e_n}(\xi,y)-U_{e_0}(\xi+\xi_0,y)\|_{L^{\infty}([-K,K]\times\R^N)}\le \frac{\varepsilon}{2}.$$
In particular, $U_{e_n}(K,y)\le \varepsilon$ and $U_{e_n}(-K,y)\ge 1-\varepsilon$ for all $y\in\R^N$ and $n$ large enough. Since $\partial_{\xi}U_e(\xi,y)<0$, it follows that
$$0\le U_{e_n}(\xi,y)\le \varepsilon \text{ for $\xi\ge K$, $y\in\R^N$ and }1-\varepsilon\le U_{e_n}(\xi,y)\le 1 \text{ for $\xi\le -K$, $y\in\R^N$}.$$
Then, we get that
$$\|U_{e_n}(\xi,y)-U_{e_0}(\xi+\xi_0,y)\|_{L^{\infty}((-\infty,-K]\cup[K,+\infty\times\R^N)}\le \varepsilon,$$
for $n$ large enough. Therefore, one can conclude that $U_{e_n}(\xi,y)\rightarrow U_{e_0}(\xi+\xi_0,y)$ uniformly in $\R\times\R^N$ as $n\rightarrow +\infty$.

Finally, we show that $\xi_0=0$. By Lemma \ref{ASY}, for any $\varepsilon>0$, there exists $K>0$ large enough such that
$$\left|\int_{[K,+\infty)\times\mathbb{T}^N} \left(U_{e_n}^2(\xi,y) -  U_{e_0}^2(\xi+\xi_0,y)\right)dyd\xi\right|<\frac{\varepsilon}{2}.$$
Since $U_{e_n}(\xi,y)\rightarrow U_{e_0}(\xi+\xi_0,y)$ uniformly in $\R\times\R^N$ as $n\rightarrow +\infty$, it follows Lebesgue's dominated convergence theorem that there is $N$ such that for $n\ge N$,
$$\left|\int_{(0,K]\times\mathbb{T}^N} \left(U_{e_n}^2(\xi,y)- U_{e_0}^2(\xi+\xi_0,y)\right)dyd\xi\right|<\frac{\varepsilon}{2}.$$
Thus, for $n\ge N$, one has that
$$\left|\int_{\R^+\times\mathbb{T}^N} \left(U_{e_n}^2(\xi,y)- U_{e_0}^2(\xi+\xi_0,y)\right)dyd\xi\right|<\varepsilon.$$
which implies
$$\int_{\R^+\times\mathbb{T}^N} U_{e_n}^2(\xi,y)dyd\xi\rightarrow\int_{\R^+\times\mathbb{T}^N} U_{e_0}^2(\xi+\xi_0,y)dyd\xi,\text{ as $n\rightarrow +\infty$}.$$
From the normalization \eqref{normalization}, it follows that
$$\int_{\R^+\times\mathbb{T}^N} U_{e_0}^2(\xi+\xi_0,y)dyd\xi=1=\int_{\R^+\times\mathbb{T}^N} U_{e_0}^2(\xi,y)dyd\xi.$$
Since $\partial_{\xi} U_{e_0}(\xi,y)<0$, that implies $\xi_0=0$. Since $e_0$ is arbitrary taken, one concludes that $U_e$ is continuous with respect to $e\in\mathbb{S}^{N-1}$ under the normalization \eqref{normalization}. The proof of Theorem \ref{th1} is thereby complete.
\end{proof}

\subsection{Differentiability}
\noindent
This section is devoted to proving the differentiability of $(U_e,c_e)$ with respect to the direction $e$.

Let us introduce some notions first. Let $L^2(\R\times\T^N)$, $H^1(\R\times\T^N)$ and $H^2(\R\times\T^N)$ be the Banach spaces defined by
\begin{eqnarray*}
\begin{aligned}
L^2(\R\times\T^N)=&\{u\in L^2_{loc}(\R\times\R^N);\ u(\xi,y+k)=u(\xi,y)\,\text{ a.e. in $\R\times\R^N$ for any $k\in \mathbb{Z}^N$},\\
&\text{ and $u\in L^2(\R\times K)$ for any bounded set $K\subset \R^N$}\},\\
H^1(\R\times\T^N)=&\{u\in H^1_{loc}(\R\times\R^N);\ u(\xi,y+k)=u(\xi,y)\,\text{ a.e. in $\R\times\R^N$ for any $k\in \mathbb{Z}^N$},\\
&\text{ and $u\in H^1(\R\times K)$ for any bounded set $K\subset \R^N$}\},
\end{aligned}
\end{eqnarray*}
and
\begin{eqnarray*}
\begin{aligned}
H^2(\R\times\T^N)=&\{u\in H^2_{loc}(\R\times\R^N);\ u(\xi,y+k)=u(\xi,y)\,\text{ a.e. in $\R\times\R^N$ for any $k\in \mathbb{Z}^N$},\\
&\text{ and $u\in H^2(\R\times K)$ for any bounded set $K\subset \R^N$}\},
\end{aligned}
\end{eqnarray*}
endowed with the norms $\|u\|_{L^2(\R\times\T^N)}=(\int_{\R}\int_{\T^N} |u|^2 dyd\xi)^{1/2}$, $$\|u\|_{H^1(\R\times\T^N)}=\|u\|_{L^2(\R\times\T^N)}+\|\partial_{\xi}u\|_{L^2(\R\times\T^N)}+\sum_{i=1}^{N}\|\partial_{y_i}u\|_{L^2(\R\times\T^N)},$$
and
$$\|u\|_{H^2(\R\times\T^N)}=\|u\|_{H^1(\R\times\T^N)}+\|\partial_{\xi\xi}u\|_{L^2(\R\times\T^N)} +\sum_{i=1}^{N}\|\partial_{\xi}\partial_{y_i}u\|_{L^2(\R\times\T^N)} +\sum_{j=1}^{N}\sum_{i=1}^{N}\|\partial_{y_j}\partial_{y_i}u\|_{L^2(\R\times\T^N)}.$$

Fix a real $\beta>0$ and for any $e\in\mathbb{S}^{N-1}$, define a linear operator
$$
M_e(v):=c_e\partial_{\xi} v+\partial_{\xi\xi} v +2\nabla_y \partial_{\xi} v\cdot e +\Delta_y v -\beta v,$$
where
$$v\in D:=\{v\in H^1(\R\times\T^N);\ \partial_{\xi\xi}v+2\nabla_y\partial_{\xi}v\cdot e+\Delta_y v                                                                          \in L^2(\R\times\T^N)\}.$$
The space $D$ is endowed with the norm $\|v\|_{D}=\|v\|_{H^1(\R\times\mathbb{T}^N)}+\|\partial_{\xi\xi}v+2\nabla_y\partial_{\xi}v\cdot e+\Delta_y v\|_{L^2(\R\times\T^N)}$.
Before going further, we need some properties of the linearization of \eqref{Ue} at $U_e$. For any $e\in\mathbb{S}^{N-1}$, define
\begin{align*}
H_e(v):=c_e\partial_{\xi} v+\partial_{\xi\xi} v +2\nabla_y \partial_{\xi} v\cdot e +\Delta_y v +f_u(y,U_e)v,\quad v\in D,
\end{align*}
and let the adjoint operator $H_e^*$ be defined by $H_e^*(u)=-c_e\partial_{\xi} u+\partial_{\xi\xi} u +2\nabla_y \partial_{\xi} u\cdot e +\Delta_y u +f_u(y,U_e)u$ for $u\in D$.

From the proofs of Lemma 3.1, Lemma 3.2 and Lemma 3.3 in \cite{DHZ}, one has the following lemma.

\begin{lemma}\label{M}
For every $e\in\mathbb{S}^{N-1}$, the operator $M_e:D\rightarrow L^2(\R\times\T^N)$ is invertible. For all $e\in\mathbb{S}^{N-1}$ and $g\in L^2(\R\times\mathbb{T}^N)$, there is a constant $C$ such that
$$
\|M_e^{-1}(g)\|_{H^1(\R\times\mathbb{T}^N)}\le C\|g\|_{L^2(\R\times\mathbb{T}^N)}.
$$
For every $e\in\mathbb{S}^{N-1}$, every $g\in L^2(\R\times\mathbb{T}^N)$ and every sequences $(e_n)_{n\in\mathbb{N}}$ in $\mathbb{S}^{N-1}$, $(g_n)_{n\in\mathbb{N}}$ in $L^2(\R\times\mathbb{T}^N)$ such that $e_n\rightarrow e$, $\|g_n-g\|_{L^2(\R\times\mathbb{T}^N)}\rightarrow 0$ as $n\rightarrow +\infty$, there holds $M_{e_n}^{-1}(g_n)\rightarrow M_e^{-1}(g)$ in $H^1(\R\times\mathbb{T}^N)$ as $n\rightarrow +\infty$.
\end{lemma}

\begin{remark}
Define
$$M_{c,e}(v):=c\partial_{\xi} v+\partial_{\xi\xi} v +2\nabla_y \partial_{\xi} v\cdot e +\Delta_y v -\beta v.$$
Following the proofs of Lemma 3.1, Lemma 3.2 and Lemma 3.3 in \cite{DHZ}, one can actually obtain that for every $e\in\mathbb{S}^{N-1}$ and $c>0$, the operator $M_{c,e}:D\rightarrow L^2(\R\times\T^N)$ is invertible and for every $e\in\mathbb{S}^{N-1}$, $c>0$, $g\in L^2(\R\times\mathbb{T}^N)$ and every sequences $(e_n)_{n\in\mathbb{N}}$ in $\mathbb{S}^{N-1}$, $(c_n)_{n\in\mathbb{N}}$ in $(0,+\infty)$ and $(g_n)_{n\in\mathbb{N}}$ in $L^2(\R\times\mathbb{T}^N)$ such that $e_n\rightarrow e$, $c_n\rightarrow c$, $\|g_n-g\|_{L^2(\R\times\mathbb{T}^N)}\rightarrow 0$ as $n\rightarrow +\infty$, there holds $M_{c_n,e_n}^{-1}(g_n)\rightarrow M_{c,e}^{-1}(g)$ in $H^1(\R\times\mathbb{T}^N)$ as $n\rightarrow +\infty$. Since $c_e$ is continuous with respect to $e\in\mathbb{S}^{N-1}$ and $\inf_{e\in\mathbb{S}^{N-1}} c_e>0$, one gets Lemma \ref{M} immediately.
\end{remark}

From the proof of Lemma 4.1 in \cite{DHZ}, one has the following lemma.

\begin{lemma}\label{lemmaH}
The operator $H_e$ and $H_e^*$ have algebraically simple eigenvalue $0$ and the range of $H_e$ is closed in $L^2(\R\times\T^N)$, and the kernel of $H_e$ is generated by $\partial_{\xi} U_e$.
\end{lemma}

For any $e\in\mathbb{S}^{N-1}$, $v\in H^2(\R\times\T^N)$, $\vartheta\in \R$ and $\eta\in\R^N$, define
$$K_e(v,\vartheta,\eta)=\vartheta\partial_{\xi} (U_e+v)+2\nabla_y\partial_{\xi} (U_e+v)\cdot \eta+f(y,U_e+v)-f(y,U_e)+\beta v,$$
and
$$
G_e(v,\vartheta,\eta):=\left(v+M_e^{-1}(K_e(v,\vartheta,\eta)),\int_{\R^+\times\T^N} \left[(U_e(\xi,y)+v(\xi,y))^2-U_e^2(\xi,y)\right]dyd\xi\right).
$$
In view of Lemma \ref{M}, the function $G_e$ maps $H^2(\R\times\T^N)\times\R\times\R^N$ to $D\times\R$. Note that $G_e(0,0,0)=0$.

\begin{lemma}\label{lemma2.10}
For every $e\in\mathbb{S}^{N-1}$, the function $G_e:\ H^2(\R\times\T^N)\times\R\times\R^N\rightarrow D\times\R$ is continuous and it is continuously Fr\'{e}chet differentiable with respect to $(v,\vartheta)$ and doubly continuously Fr\'{e}chet differentiable with respect to $\eta$.
\end{lemma}

\begin{proof}
Since $K_e$ is affine with respect to $\vartheta$ and $\eta$ and the function $f(y,u)$ is globally Lipschitz-continuous in $u$ uniformly for $y\in\T^N$, it is elementary to get the continuity of $K_e$. Then, from lemma \ref{M}, one has that $G_1(v,\vartheta,\eta):=v+M_e^{-1}(K_e(v,\vartheta,\eta))$ is continuous in $H^2(\R\times\T^N)\times\R\times\R^N$. Since the continuity of $G_2:=\int_{\R^+\times\T^N} \left[(U_e(\xi,y)+v(\xi,y))^2-U_e^2(\xi,y)\right]dyd\xi$ is obvious from Cauchy-Schwarz inequality, it follows that $G_e=(G_1,G_2)$ is continuous in $H^1(\R\times\T^N)\times\R\times\R^N$.

Since $G_e$ is affine with respect to $\eta$, it is obvious that $G_e$ is doubly continuously Fr\'{e}chet differentiable with respect to $\eta$ and the first ordered derivative is
$$\partial_{\eta} G_e(v,\vartheta,\eta)\tilde{\eta}=\left(M_e^{-1}(2\nabla_y\partial_{\xi} (U_e+v)\cdot\tilde{\eta}),0\right),$$
for any $(v,\vartheta,\eta)\in H^2(\R\times\mathbb{T}^N)\times\R\times\R^N$ and $\tilde{\eta}\in\R$. Now we show that $G_e$ is continuously Fr\'{e}chet differentiable with respect to $(v,\vartheta)$. Notice that $f(y,U_e+u)$ is continuously Fr\'{e}chet differentiable with respect to $u$. In fact, for any $u$, $v\in H^2(\R\times\T^N)$, one has that
$$
\lim_{h\rightarrow 0} \frac{f(y,U_e+u+hv)-f(y,U_e+u)}{h}=f_u(y,U_e+u)v,
$$
in $L^2(\R\times\mathbb{T}^{N})$. Hence, the function $G_e(v,\vartheta,\eta)$ is Fr\'{e}chet differentiable with respect to $(v,\vartheta)$ with derivative
\be\label{2.14}
\begin{aligned}
&\partial_{(v,\vartheta)} G_e(v,\vartheta,\eta)(\tilde{v},\tilde{\vartheta})\\
&=\left(\baa{c}
\tilde{v}+M_e^{-1}(\vartheta\partial_{\xi} \tilde{v}+\tilde{\vartheta}\partial_{\xi}(U_e+v)+2\nabla_y\partial_{\xi} \tilde{v}\cdot \eta +f_u(y,U_e+v)\tilde{v} +\beta\tilde{v})\\
2\int_{\R^+\times\T^N} (U_e(\xi,y)+v(\xi,y))\tilde{v}(\xi,y)dyd\xi
\eaa
\right).\end{aligned}
\ee
for any $(v,\vartheta,\eta)\in H^2(\R\times\mathbb{T}^N)\times\R\times\R^N$ and $(\tilde{v},\tilde{\theta})\in H^2(\R\times\mathbb{T}^N)\times\R$.
 Since $f_u(y,u)$ is globally Lipschitz-continuous in $u$ uniformly for $y\in\T^N$ and following the arguments in the first paragraph, one gets that $\partial_{(v,\vartheta)} G_e: H^2(\R\times\T^N)\times\R\times\R^N\rightarrow \mathcal{L}(H^2(\R\times\T^N)\times\R, D\times\R)$ is continuous.

This completes the proof.
\end{proof}
\vskip 0.3cm

For any $e\in\mathbb{S}^{N-1}$ and $(\tilde{v},\tilde{\vartheta})\in D\times\R$, define
\begin{align}\label{Q}
Q_e(\tilde{v},\tilde{\vartheta})=\left(\tilde{v}+M_e^{-1}(\tilde{\vartheta}\partial_{\xi}U_e +f_u(y,U_e)\tilde{v} +\beta\tilde{v}),
2\int_{\R^+\times\T^N} U_e(\xi,y)\tilde{v}(\xi,y)dyd\xi\right).
\end{align}
Notice that $Q_e$ has the same form as $\partial_{(v,\vartheta)} G_e(0,0,0)$ from \eqref{2.14}.

\begin{lemma}\label{lemma2.11}
For every $e\in\mathbb{S}^{N-1}$, the operator $Q_e: D\times\R\rightarrow D\times\R$ is invertible. Then, for every $e\in\mathbb{S}^{N-1}$, $g\in D$, $d\in\R$ and every sequences $(e_n)_{n\in\mathbb{N}}$ in $\mathbb{S}^{N-1}$, $(g_n)_{n\in\mathbb{N}}$ in $D$, $(d_n)_{n\in\mathbb{N}}$ in $\R$ such that $e_n\rightarrow e$, $\|g_n-g\|_{D}\rightarrow 0$ and $|d_n-d|\rightarrow 0$ as $n\rightarrow +\infty$, there holds $Q_{e_n}^{-1}(g_n,d_n)\rightarrow Q_e^{-1}(g,d)$ in $L^2(\R\times\mathbb{T}^N)\times\R$ as $n\rightarrow +\infty$, where the space $L^2(\R\times\mathbb{T}^N)\times\R$ is endowed with the norm $\|(\tilde{v},\tilde{\vartheta})\|_{L^2(\R\times\mathbb{T}^N)\times\R}=\|\tilde{v}\|_{L^2(\R\times\mathbb{T}^N)}+|\tilde{\vartheta}|$. Furthermore, for all $e\in\mathbb{S}^{N-1}$, $g\in L^2(\R\times\mathbb{T}^N)$ and $d\in\R$, there is $C>0$ such that
$$\|Q_e^{-1}(g,d)\|_{L^2(\R\times\mathbb{T}^N)\times\R}\le C\|(g,d)\|_{D\times\R}.$$
\end{lemma}

\begin{proof}
The proof of invertibility can just follow the proof of \cite[Lemma 3.3]{DHZ} step by step, by only noticing that the kernel of $H_e$ is generated by $\partial_{\xi} U_e$ from Lemma \ref{lemmaH} and the domain of $Q_e$ is $D\times\R$.

Now, we prove the convergence. Since $Q_e^{-1}(g,d)$ is linear for $(g,d)\in D\times\R$, we first show that $Q_e^{-1}(g_n,d_n)\rightarrow (0,0)$ in $L^2(\R\times\T^N)\times\R$ as $n\rightarrow +\infty$ when $\|g_n\|_D\rightarrow 0$ and $|d_n|\rightarrow 0$ as $n\rightarrow +\infty$. Let $(\tilde{v}_n,\tilde{\vartheta}_n)=Q_e^{-1}(g_n,d_n)$. Since the range of $Q_e$ is closed and the kernel of $Q_e$ is trivial, one has that $(\tilde{v}_n,\tilde{\vartheta}_n)\rightarrow (0,0)$ in $L^2(\R\times\T^N)\times\R$ (actually $\tilde{v}_n\rightarrow 0$ strongly in $L^2(\R\times\mathbb{T}^N)$, weakly in $H^1$). Moreover, by Lemma \ref{M}, one has that $Q_{e_n}^{-1}(g,d)\rightarrow Q_e^{-1}(g,d)$ in $L^2(\R\times\mathbb{T}^N)\times\R$ as $n\rightarrow +\infty$ when $e_n\rightarrow e$ as $n\rightarrow +\infty$ for any $g\in D$ and $d\in\R$. Since $\|Q_{e_n}^{-1}(g_n,d_n)- Q_e^{-1}(g,d)\|_{L^2(\R\times\mathbb{T}^N)\times\R}\le \|Q_{e_n}^{-1}(g_n,d_n)- Q_e^{-1}(g_n,d_n)\|_{L^2(\R\times\mathbb{T}^N)\times\R}+\|Q_{e}^{-1}(g_n,d_n)- Q_e^{-1}(g,d)\|_{L^2(\R\times\mathbb{T}^N)\times\R}$, one can get the conclusion that $Q_{e_n}^{-1}(g_n,d_n)\rightarrow Q_e^{-1}(g,d)$ in $L^2(\R\times\mathbb{T}^N)\times\R$ as $n\rightarrow +\infty$, when $e_n\rightarrow e$, $\|g_n-g\|_{D}\rightarrow 0$ and $|d_n-d|\rightarrow 0$ as $n\rightarrow +\infty$.

For every $e\in\mathbb{S}^{N-1}$ and any $g\in D$, $d\in\R$, there is $\delta_e>0$ small enough such that
$$\|\frac{\delta_e}{\|(g,d)\|_{D\times\R}}Q_e^{-1}(g,d)\|_{L^2(\R\times\T^N)\times\R}\le 1,$$
since $Q_e^{-1}(g_n,d_n)\rightarrow (0,0)$ in $L^2(\R\times\T^N)\times\R$ as $n\rightarrow +\infty$ when $\|g_n\|_D\rightarrow 0$ and $|d_n|\rightarrow 0$ as $n\rightarrow +\infty$. That implies that for every $e\in\mathbb{S}^{N-1}$, there is $\delta_e>0$ such that
\be\label{eq2.20}
\|Q_e^{-1}(g,d)\|_{L^2(\R\times\T^N)\times\R}\le \frac{1}{\delta_e}\|(g,d)\|_{D\times\R}.
\ee
We now show that $1/{\delta_e}$ is uniformly bounded for $e\in\mathbb{S}^{N-1}$. Assume by contradiction that there is a sequence $(e_n)_{n\in\mathbb{N}}\subset \mathbb{S}^{N-1}$ such that
$$\|\frac{1}{\|(g,d)\|_{D\times\R}}Q_{e_n}^{-1}(g,d)\|_{L^2(\R\times\T^N)\times\R}\rightarrow +\infty, \text{ as $n\rightarrow +\infty$}.$$
There is $e_0\in \mathbb{S}^{N-1}$ such that $e_n\rightarrow e_0$, up to a subsequence, as $n\rightarrow +\infty$. Then, up to a subsequence,
$Q_{e_n}^{-1}(g,d)\rightarrow Q_{e_0}^{-1}(g,d)$ in $L^2(\R\times\T^N)\times\R$ as $n\rightarrow +\infty$.
Thus, one has that
$$\|\frac{1}{\|(g,d)\|_{D\times\R}}Q_{e_0}^{-1}(g,d)\|_{L^2(\R\times\T^N)\times\R}= +\infty,$$
which contradicts \eqref{eq2.20}. Therefore, for all $e\in\mathbb{S}^{N-1}$, $g\in L^2(\R\times\mathbb{T}^N)$ and $d\in\R$, there is $C>0$ such that
$$\|Q_e^{-1}(g,d)\|_{L^2(\R\times\mathbb{T}^N)\times\R}\le C\|(g,d)\|_{D\times\R}.$$

The proof is thereby complete.
\end{proof}
\vskip 0.3cm

Given the previous lemmas, we are now ready to prove Theorem \ref{th2}.
\vskip 0.3cm

\begin{proof}[Proof of Theorem \ref{th2}]
{\it Step 1: first order differentiability.} For every $e\in\mathbb{S}^{N-1}$, normalize $U_e$ by
\be\label{eq2.19}
\int_{\R^+\times\mathbb{T}^{N}} U_e^2(\xi,y)dyd\xi=1.
\ee
For any $b\in \R^N\setminus\{0\}$, let
$$U_b=U_{\frac{b}{|b|}}\text{ and } c_b=c_{\frac{b}{|b|}}.$$
Then, by Theorem \ref{th1}, $(U_b,c_b)$ is well defined and continuous with respect to $b\in\R^N\setminus\{0\}$. Furthermore, $U_b$ and $c_b$ satisfy
\be\label{eq2.21}
c_{b} \partial_{\xi} U_{b} +\partial_{\xi\xi} U_{b} +2\nabla_y \partial_{\xi} U_{b}\cdot \frac{b}{|b|}+\Delta_y U_{b} +f(y,U_{b})=0.
\ee

Now fix arbitrary $e\in\mathbb{S}^{N-1}$. For any $h\in\R^N$ such that $e+h\in \R^N\setminus\{0\}$, one has that $U_{e+h}$ and $c_{e+h}$ satisfy \eqref{eq2.21} with $b$ replaced by $e+h$. Let $\tilde{U}_h=U_{e+h}-U_e\in D$, $\tilde{c}_h=c_{e+h}-c_e\in\R$ and $\tilde{h}=(e+h)/|e+h|-e$. Notice that $\|(\tilde{U}_h,\tilde{c}_h)\|_{L^2(\R\times\mathbb{T}^{N})\times\R}\rightarrow 0$ and $\tilde{h}=-(e\cdot h)e+h+o(|h|)$ as $|h|\rightarrow 0$. By the normalization \eqref{eq2.19}, $(U_{e+h},c_{e+h})$ satisfying \eqref{eq2.21} with $b=e+h$ and $(U_e,c_e)$ satisfying \eqref{Ue}, one can compute that
$$G_e(\tilde{U}_h,\tilde{c}_h,\tilde{h})=(0,0).$$
Recalling that $G_e(0,0,0)=(0,0)$ and by Lemma \ref{lemma2.10} and the definition of Fr\'{e}chet differentiability, it follows that
$$(0,0)=G_e(\tilde{U}_h,\tilde{c}_h,\tilde{h})-G_e(0,0,0)=\partial_{(v,\vartheta)} G_e(0,0,0)(\tilde{U}_h,\tilde{c}_h)+\partial_{\eta} G_e(0,0,0)\tilde{h}+\omega_1(\tilde{h})+\omega_2(\tilde{U}_h,\tilde{c}_h),$$
where $\omega_1(\tilde{h})=o(|h|)$ and $\omega_2(\tilde{U}_h,\tilde{c}_h)=o(\|(\tilde{U}_h,\tilde{c}_h)\|_{L^2(\R\times\mathbb{T}^{N})\times\R})$  as $|h|\rightarrow 0$. Since  $\partial_{(v,\vartheta)} G_e(0,0,0)$ has the same form as $Q_e$ and $\tilde{U}_h\in D$, $\tilde{c}_h\in \R$, one can replace $\partial_{(v,\vartheta)} G_e(0,0,0)$ by $Q_e$ in the above equation. Thus, it follows from Lemma \ref{lemma2.11} that
\begin{align}
(\tilde{U}_h,\tilde{c}_h)+Q_e^{-1}(\omega_2(\tilde{U}_h,\tilde{c}_h))=& -Q_e^{-1}(\partial_{\eta} G_e(0,0,0)\tilde{h})-Q_e^{-1}(\omega_1(\tilde{h}))\nonumber\\
=& -Q_e^{-1}(M_e^{-1}(2\nabla_y\partial_{\xi} U_e\cdot \tilde{h}),0)-Q_e^{-1}(\omega_1(\tilde{h})).\label{2.21}
\end{align}
Then, one has that
$$\frac{1}{|h|}\|(\tilde{U}_h,\tilde{c}_h)+Q_e^{-1}(\omega_2(\tilde{U}_h,\tilde{c}_h))\|_{L^2(\R\times\mathbb{T}^N)\times\R}
= \frac{1}{|h|}\|Q_e^{-1}(M_e^{-1}(2\nabla_y\partial_{\xi} U_e\cdot \tilde{h}),0)+Q_e^{-1}(\omega_1(\tilde{h}))\|_{L^2(\R\times\mathbb{T}^N)\times\R}.$$
By Lemma \ref{M}, Lemma \ref{lemma2.11} and $\omega_1(\tilde{h})=o(|h|)$ as $|h|\rightarrow 0$, the right hand is bounded as $|h|\rightarrow 0$. Moreover, since $\omega_2(\tilde{U}_h,\tilde{c}_h)=o(\|(\tilde{U}_h,\tilde{c}_h)\|_{L^2(\R\times\mathbb{T}^{N})\times\R})$ as $|h|\rightarrow 0$, one has that
\begin{align*}
\|(\tilde{U}_h,\tilde{c}_h)+Q_e^{-1}(\omega_2(\tilde{U}_h,\tilde{c}_h))\|_{L^2(\R\times\mathbb{T}^N)\times\R}\ge& \|(\tilde{U}_h,\tilde{c}_h)\|_{L^2(\R\times\mathbb{T}^N)\times\R}-\|Q_e^{-1}(\omega_2(\tilde{U}_h,\tilde{c}_h))\|_{L^2(\R\times\mathbb{T}^N)\times\R}\\
\ge& \frac{1}{2}\|(\tilde{U}_h,\tilde{c}_h)\|_{L^2(\R\times\mathbb{T}^N)\times\R},
\end{align*}
as $|h|\rightarrow 0$. Then, $\|(\tilde{U}_h,\tilde{c}_h)\|_{L^2(\R\times\mathbb{T}^N)\times\R}/|h|$ is bounded as $|h|\rightarrow 0$. It implies that $Q_e^{-1}(\omega_2(\tilde{U}_h,\tilde{c}_h))=o(|h|)$ as $|h|\rightarrow 0$. Therefore, by \eqref{2.21} and recalling that $\tilde{h}=-(e\cdot h)e+h+o(|h|)$ as $|h|\rightarrow 0$, one gets that
\begin{align*}
(U_{e+h}-U_e,c_{e+h}-c_e)&=(\tilde{U}_h,\tilde{c}_h)= -Q_e^{-1}(M_e^{-1}(2\nabla_y\partial_{\xi} U_e\cdot \tilde{h}),0)+o(|h|)\\
&=(e\cdot h)Q_e^{-1}(M_e^{-1}(2\nabla_y\partial_{\xi} U_e\cdot e),0) - Q_e^{-1}(M_e^{-1}(2\nabla_y\partial_{\xi} U_e\cdot h),0)+o(|h|).
\end{align*}
Thus, by the arbitrariness of $e$ in $\mathbb{S}^{N-1}$, one can conclude that $(U_b,c_b)$ is Fr\'{e}chet differentiable everywhere at $e\in\mathbb{S}^{N-1}$. Denote the derivative by $(U_e',c_e')$, that is, for any $h\in\R^N$
\be\label{U'}
(U_e'(h),c_e'(h))=(e\cdot h)Q_e^{-1}(M_e^{-1}(2\nabla_y\partial_{\xi} U_e\cdot e),0) - Q_e^{-1}(M_e^{-1}(2\nabla_y\partial_{\xi} U_e\cdot h),0),
\ee
where $(U_e',c_e'):\R^N\rightarrow L^2(\R\times\mathbb{T}^N)\times \R$.
By Lemma \ref{M}, Lemma \ref{lemma2.11} and the continuity of $U_e$ with respect to $e\in\mathbb{S}^{N-1}$, one has that for any $h\in\R^N$, $(U_e'(h),c_e'(h))$ is continuous with respect to $e\in\mathbb{S}^{N-1}$ (one can actually prove that $(U_{e_n}'(h),c_{e_n}'(h))\rightarrow (U'_e(h),c'_e(h))$ as $n\rightarrow +\infty$ when $e_n\rightarrow e$ as $n\rightarrow +\infty$). Since $U_e(\cdot,\cdot)\in C^{2,2}(\R\times\R^N)$, it implies that $U_e'(h)(\cdot,\cdot)$ is in $C^{2,2}(\R\times\R^N)$, for every $h\in \R^N$.

Then, for any $b\in\R^N\setminus\{0\}$ and any direction $h\in\R^N$, one gets that
\begin{align*}
(U_{b+h}-U_b,c_{b+h}-c_b)&=\left(U_{\frac{b+h}{|b+h|}}-U_{\frac{b}{|b|}},c_{\frac{b+h}{|b+h|}}-c_{\frac{b}{|b|}}\right)\\
&=\left(U'_{\frac{b}{|b|}}(\frac{h}{|b|}-\frac{b\cdot h}{|b|^3}b),c'_{\frac{b}{|b|}}(\frac{h}{|b|}-\frac{b\cdot h}{|b|^3}b)\right)+o(|h|).
\end{align*}
This implies that $(U_b,c_b)$ is continuously Fr\'{e}chet differentiable at any $b\in\R^N\setminus\{0\}$.

{\it Step 2: second order differentiability.} By Step 1, $(U'_b,c'_b)$ is well defined and continuous with respect to $b\in\R^N\setminus\{0\}$. Fix arbitrary $e\in\mathbb{S}^{N-1}$ and $h\in\R^{N}$. From the definition of $(U_b,c_b)$, one has that $(U_b,c_b)$ satisfies \eqref{eq2.21}.
Differentiating \eqref{eq2.21} at $b$ on the direction $h\in\R^N$, one gets that
\begin{align}
&c'_b(h)\partial_{\xi} U_b +c_b\partial_{\xi} U'_b(h) +\partial_{\xi\xi} U'_b(h) +2\nabla_y\partial_{\xi} U_b\cdot (\frac{h}{|b|}-\frac{b\cdot h}{|b|^3}b)\nonumber\\
&+2\nabla_y\partial_{\xi} U'_e(h)\cdot \frac{b}{|b|} +\Delta_y U'_b(h) +f_u(y,U_b)U'_b(h)=0\label{2.23}
\end{align}

For any $e\in\mathbb{S}^{N-1}$, $h\in\R^N$, $v_1$, $v_2\in H^2(\R\times\T^N)$, $\vartheta_1$, $\vartheta_2\in \R$ and $\eta\in\R^N$, define
\begin{align*}
K'_e(v_1,&\vartheta_1,v_2,\vartheta_2,\eta):=c'_e(h)\partial_{\xi}v_1+\vartheta_2\partial_{\xi} (U_e+v_1)+\vartheta_1\partial_{\xi} (U'_e(h)+v_2) \\
&+2\nabla_y\partial_{\xi} U_e\cdot [\frac{h}{|e+\eta|}-h-\frac{(e+\eta)\cdot h}{|e+\eta|^3}(e+\eta)+(e\cdot h)e]\\
&+2\nabla_y\partial_{\xi} v_1\cdot [\frac{h}{|e+\eta|}-\frac{(e+\eta)\cdot h}{|e+\eta|^3}(e+\eta)]+2\nabla_y\partial_{\xi} U'_e(h)\cdot (\frac{e+\eta}{|e+\eta|}-e)\\
&+2\nabla_y\partial_{\xi} v_2\cdot (\frac{e+\eta}{|e+\eta|}-e) +\beta v_2 +f_u(y,U_e+v_1)(U'_e(h)+v_2)-f_u (y,U_e)U'_e(h)
\end{align*}
and
\begin{align*}
&G'_e(v_1,\vartheta_1,v_2,\vartheta_2,\eta)\\
&:=\left(v_2+M_e^{-1}(K'_e(v_1,\vartheta_1,v_2,\vartheta_2,\eta)),2\int_{\R^+\times\T^N} \left[U'_e(h)(\xi,y)v_1(\xi,y)+v_2(U_e(\xi,y)+v_1)\right]dyd\xi\right).
\end{align*}
Following the arguments of Lemma \ref{lemma2.10}, one has that for every $e\in\mathbb{S}^{N-1}$, the function $G'_e: H^2(\R\times\T^N)\times\R\times H^2(\R\times\T^N)\times\R\times\R^N\rightarrow D\times\R$ is continuous and it is continuously Fr\'{e}chet differentiable with respect to $(v_1,\vartheta_1)$ and $(v_2,\vartheta_2)$ respectively, and doubly continuously Fr\'{e}chet differentiable with respect to $\eta$.  One can compute that the function $G'_e(v_1,v_2,\vartheta_1,\vartheta_2,\eta)$ is with derivatives
$$\partial_{\eta} G'_e(v_1,\vartheta_1,v_2,\vartheta_2,\eta)\tilde{\eta}=\left(
M_e^{-1}(J_1),0\right),$$
$$
\partial_{(v_1,\vartheta_1)} G'_e(v_1,\vartheta_1,v_2,\vartheta_2,\eta)(\tilde{v}_1,\tilde{\vartheta}_1)=\left(
M_e^{-1}(J_2),
2\int_{\R^+\times\T^N} \left(U'_e(h)\tilde{v}_1+v_2\tilde{v}_1\right)dyd\xi\right),
$$
$$
\partial_{(v_2,\vartheta_2)} G'_e(v_1,\vartheta_1,v_2,\vartheta_2,\eta)(\tilde{v}_2,\tilde{\vartheta}_2)
=\left(
\tilde{v}_2+M_e^{-1}(J_3),
2\int_{\R^+\times\T^N} \tilde{v}_2(U_e+v_1)dyd\xi\right),
$$
where
$$
J_1=2\nabla_y\partial_{\xi} (U'_e(h)+v_2)\cdot \eta_1 +2\nabla_y\partial_{\xi} (U_e+v_1)\cdot \eta_2,
$$
$$J_2=(c'_e(h)+\vartheta_2)\partial_{\xi} \tilde{v}_1+2\nabla_y\partial_{\xi} \tilde{v}_1\cdot [\frac{h}{|e+\eta|}-\frac{(e+\eta)\cdot h}{|e+\eta|^3}(e+\eta)] +f_{uu}(y,U_e+v_1)\tilde{v}_1(U'_e(h)+v_2),$$
$$
J_3=\vartheta_1\partial_{\xi} \tilde{v}_2+2\nabla_y\partial_{\xi} \tilde{v}_2\cdot (\frac{e+\eta}{|e+\eta|}-e) +\beta \tilde{v}_2 +f_u(y,U_e+v_1)\tilde{v}_2 +\tilde{\vartheta}_2 \partial_{\xi} (U_e+v_1),
$$
with
$$\eta_1=\frac{\tilde{\eta}}{|e+\eta|}-\frac{(e+\eta)\cdot \tilde{\eta}}{|e+\eta|^3}(e+\eta),$$
$$\eta_2=-\frac{(e+\eta)\cdot \tilde{\eta}}{|e+\eta|^3}h-\frac{(e+\eta)\cdot h}{|e+\eta|^3}\tilde{\eta}-[\frac{\tilde{\eta}\cdot h}{|e+\eta|^3}-3(e+\eta)\cdot h\frac{(e+\eta)\cdot \tilde{\eta}}{|e+\eta|^4}](e+\eta),$$
for any $(v_1,\vartheta_1,v_2,\vartheta_2,\eta)\in H^2(\R\times\T^N)\times\R\times H^2(\R\times\T^N)\times\R\times\R^N$, $\tilde{\eta}\in\R^N$, $(\tilde{v}_1,\tilde{\theta}_1)\in H^2(\R\times\mathbb{T}^N)\times\R$ and $(\tilde{v}_2,\tilde{\theta}_2)\in H^2(\R\times\mathbb{T}^N)\times\R$. One also has that
\be
\baa{l}
\partial_{(v_2,\vartheta_2)} G'_e(0,0,0,0,0)(\tilde{v}_2,\tilde{\vartheta}_2)\\
=\left(
\tilde{v}_2+M_e^{-1}(\beta \tilde{v}_2 +f_u(y,U_e)\tilde{v}_2 +\tilde{\vartheta}_2 \partial_{\xi} U_e),
2\int_{\R^+\times\T^N} U_e(\xi,y)\tilde{v}_2(\xi,y)dyd\xi
\right).\eaa
\ee
Notice that $\partial_{(v_2,\vartheta_2)} G'_e(0,0,0,0,0)$ has the same form as $Q_e$.

For any $\rho\in\R^N$ such that $e+\rho\in\R^N\setminus\{0\}$, let $\tilde{U}'_{\rho}(h)=U'_{e+{\rho}}(h)-U'_e(h)\in D$, $\tilde{c}'_{\rho}(h)=c'_{e+\rho}(h)-c'_e(h)\in\R$, $\tilde{U}_{\rho}=U_{e+\rho}-U_e\in D$, and $\tilde{c}_{\rho}=c_{e+\rho}-c_e\in\R$. Then, from \eqref{2.23}, it follows that $G'_e(\tilde{U}_{\rho},\tilde{c}_{\rho},\tilde{U}'_{\rho}(h),\tilde{c}'_{\rho}(h),\rho)=0$. By $G(0,0,0,0,0)=(0,0)$, it follows that
\begin{align*}
(0,0)=&G(\tilde{U}_{\rho},\tilde{c}_{\rho},\tilde{U}'_{\rho}(h),\tilde{c}'_{\rho}(h),\rho)-G(0,0,0,0,0)\\
=&\partial_{(v_1,\vartheta_1)} G(0,0,0,0,0)(\tilde{U}_{\rho},\tilde{c}_{\rho})+\partial_{(v_2,\vartheta_2)} G(0,0,0,0,0)(\tilde{U}'_{\rho}(h),\tilde{c}'_{\rho}(h))\\
&+\partial_{\eta} G(0,0,0,0,0)\rho+\omega_1(\rho)+\omega_2(\tilde{U}_{\rho},\tilde{c}_{\rho})+\omega_3(\tilde{U}'_{\rho}(h),\tilde{c}'_{\rho}(h)),
\end{align*}
where $\omega_1(\rho)=o(|\rho|)$, $\omega_2(\tilde{U}_{\rho},\tilde{c}_{\rho})=o(\|(\tilde{U}_{\rho},\tilde{c}_{\rho})\|_{L^2(\R\times\mathbb{T}^{N})\times\R})$ (remember that $\|(\tilde{U}_{\rho},\tilde{c}_{\rho})\|_{L^2(\R\times\mathbb{T}^{N})\times\R}=O(|\rho|)$ from arguments of Step 1) and $\omega_2(\tilde{U}'_{\rho}(h),\tilde{c}'_{\rho}(h))=o(\|(\tilde{U}'_{\rho}(h),\tilde{c}'_{\rho}(h))\|_{L^2(\R\times\mathbb{T}^{N})\times\R})$ as $|\rho|\rightarrow 0$.  Since  $\partial_{(v_2,\vartheta_2)} G(0,0,0,0,0)$ has the same form as $Q_e$ and $\tilde{U}'_h(h)\in D$, $\tilde{c}'_h(h)\in \R$, one can replace $\partial_{(v_2,\vartheta_2)} G(0,0,0,0,0)$ by $Q_e$ in the above equation. Thus, it follows from Lemma \ref{lemma2.11} that
\begin{align}
&(\tilde{U}'_{\rho}(h),\tilde{c}'_{\rho}(h))+Q_e^{-1}(\omega_3(\tilde{U}'_{\rho}(h),\tilde{c}'_{\rho}(h)))\nonumber\\
=& -Q_e^{-1}(\partial_{\eta} G(0,0,0,0,0)\rho)-Q_e^{-1}(\omega_1(\rho)+\omega_2(\tilde{U}_{\rho},\tilde{c}_{\rho}))\nonumber\\
=& -Q_e^{-1}(M_e^{-1}(2\nabla_y\partial_{\xi} U'_e(h)\cdot \rho_1+2\nabla_y\partial_{\xi} U_e\cdot \rho_2),0)-Q_e^{-1}(\omega_1(\rho)+\omega_2(\tilde{U}_{\rho},\tilde{c}_{\rho})),\label{eq2.26}
\end{align}
with
\be\label{rho1rho2}
\rho_1=\rho-(e\cdot \rho)e,\quad \rho_2=-(e\cdot \rho)h-(e\cdot h)\rho-[\rho\cdot h-3(e\cdot h)(e\cdot \rho)]e.
\ee
Then, one has that
\begin{align*}
&\frac{1}{|\rho|}\|(\tilde{U}'_{\rho}(h),\tilde{c}'_{\rho}(h))+Q_e^{-1}(\omega_3(\tilde{U}'_{\rho}(h),\tilde{c}'_{\rho}(h)))\|_{L^2(\R\times\mathbb{T}^N)\times\R}\\
=& \frac{1}{|\rho|}\|Q_e^{-1}(M_e^{-1}(2\nabla_y\partial_{\xi} U'_e(h)\cdot \rho_1+2\nabla_y\partial_{\xi} U_e\cdot \rho_2),0)+Q_e^{-1}(\omega_1(\rho)+\omega_2(\tilde{U}_{\rho},\tilde{c}_{\rho}))\|_{L^2(\R\times\mathbb{T}^N)\times\R}.
\end{align*}
Since $\omega_1(\rho)+\omega_2(\tilde{U}_{\rho},\tilde{c}_{\rho})=o(|\rho|)$ as $|\rho|\rightarrow 0$, the right hand is bounded as $|\rho|\rightarrow 0$. Moreover, since $\omega_3(\tilde{U}'_{\rho}(h),\tilde{c}'_{\rho}(h)=o(\|(\tilde{U}'_{\rho}(h),\tilde{c}'_{\rho}(h))\|_{L^2(\R\times\mathbb{T}^{N})\times\R})$ as $|\rho|\rightarrow 0$, one has that
\begin{align*}
&\|(\tilde{U}'_{\rho}(h),\tilde{c}'_{\rho}(h))+Q_e^{-1}(\omega_3(\tilde{U}'_{\rho}(h),\tilde{c}'_{\rho}(h)))\|_{L^2(\R\times\mathbb{T}^N)\times\R}\\
\ge& \|(\tilde{U}'_{\rho}(h),\tilde{c}'_{\rho}(h))\|_{L^2(\R\times\mathbb{T}^N)\times\R}-\|Q_e^{-1}(\omega_3(\tilde{U}'_{\rho}(h),\tilde{c}'_{\rho}(h)))\|_{L^2(\R\times\mathbb{T}^N)\times\R}\\
\ge& \frac{1}{2}\|(\tilde{U}'_{\rho}(h),\tilde{c}'_{\rho}(h))\|_{L^2(\R\times\mathbb{T}^N)\times\R},
\end{align*}
as $|\rho|\rightarrow 0$. Then, $\|(\tilde{U}'_{\rho}(h),\tilde{c}'_{\rho}(h))\|_{L^2(\R\times\mathbb{T}^N)\times\R}/|\rho|$ is bounded as $|\rho|\rightarrow 0$. Thus, by Lemma \ref{lemma2.11}, one has that $Q_e^{-1}(\omega_3(\tilde{U}'_{\rho}(h),\tilde{c}'_{\rho}(h)))=o(|\rho|)$ as $|\rho|\rightarrow 0$. Therefore, by \eqref{eq2.26}, one gets that
\begin{eqnarray*}
\begin{aligned}
(U'_{e+\rho}(h)&-U'_e(h),c'_{e+\rho}(h)-c'_e(h))=(\tilde{U}'_{\rho}(h),\tilde{c}'_{\rho}(h))\\
=& -Q_e^{-1}(M_e^{-1}(2\nabla_y\partial_{\xi} U'_e(h)\cdot \rho_1+2\nabla_y\partial_{\xi} U_e\cdot \rho_2),0)+o(|\rho|)
\end{aligned}
\end{eqnarray*}
Thus, by the arbitrariness of $e\in\mathbb{S}^{N-1}$, one can conclude that $(U'_b(h),c'_b(h))$ is Fr\'{e}chet differentiable at $e\in\mathbb{S}^{N-1}$ for any $h\in\R^N$. Denote the derivative by $(U''_e(h),c''_e(h))$, that is, for any $\rho\in\R^N$
\begin{align}
(U''_e(h)(\rho),c''_e(h)(\rho))
=-Q_e^{-1}(M_e^{-1}(2\nabla_y\partial_{\xi} U'_e(h)\cdot \rho_1+2\nabla_y\partial_{\xi} U_e\cdot \rho_2),0),\label{U''}
\end{align}
where $\rho_1$, $\rho_2$ are defined in \eqref{rho1rho2}.
By Lemma \ref{M}, Lemma \ref{lemma2.11} and the continuity of $U'_e(h)$ with respect to $e\in\mathbb{S}^{N-1}$, one has that for any $h\in\R^N$ and $\rho\in\R^N$, $(U''_e(h)(\rho),c''_e(h)(\rho))$ is continuous with respect to $e\in\mathbb{S}^{N-1}$. Since $U'_e(h)(\cdot,\cdot)\in C^{2,2}(\R\times\R^N)$, it implies that $U''_e(h)(\rho)(\cdot,\cdot)$ is in $C^{2,2}(\R\times\R^N)$.

Similarly as in Step 1, one can also get that $U'_b(h)$ is continuously Fr\'{e}chet differentiable at any $b\in\R^N\setminus\{0\}$. The proof is thereby complete.
\end{proof}
\vskip 0.3cm

From the arguments above, we know that for every $e\in\mathbb{S}^{N-1}$,  $U'_e$, $\partial_{\xi} U'_e$, $\partial_{y_i} U'_e$ ($i=1,\ \cdots,\ N$) and $U''_e$ are bounded linear operators. We emphasize the meaning of the Fr\'{e}chet derivatives at $e\in\mathbb{S}^{N-1}$ by examples that $U'_e(h)(\cdot,\cdot)$ is the derivative of $U_b(\cdot,\cdot)$ (where $U_b(\cdot,\cdot)$ is defined in \eqref{Ub}) at $e\in\mathbb{S}^{N-1}$ on the direction $h\in\R^N$ and $U''_e(h)(\rho)(\cdot,\cdot)$ is the derivative of $U'_b(h)(\cdot,\cdot)$ at $e\in\mathbb{S}^{N-1}$ on the direction $\rho\in\R^N$. As we mentioned in the proof of Theorem \ref{th2} that $U'_e(h)(\cdot,\cdot)$ is in $C^{2,2}(\R\times\R^N)$, the derivatives of $U'_e(h)(\cdot,\cdot)$ with respect to $\xi$ and $y$ are well defined, denoted by $\partial_{\xi} U'_e(h)$, $\partial_{y_i} U'_e(h)$ ($i=1,\ \cdots,\ N$) for any $h\in\R^N$. By the definition of $U'_e$ and the definition of Fr\'{e}chet differentiability, we know that $\partial_{\xi} U'_e(h)$, $\partial_{y_i} U'_e(h)$ ($i=1,\ \cdots,\ N$) are also the Fr\'{e}chet derivatives of $\partial_{\xi} U_b$ and $\partial_{y_i} U_b$ ($i=1,\ \cdots,\ N$) at $e\in\mathbb{S}^{N-1}$ on the direction $h\in\R^N$. Furthermore, since $U'_e(h)$ is a linear operator with respect to $h$, we can easily get that $U'_e(h)$ is Fr\'{e}chet differentiable with respect to $h$, with the derivative $U'_e(\rho)$ at any $h\in\R^N$ on the direction $\rho\in\R^N$. Then, we denote the norm of the Fr\'{e}chet derivatives by that for every $e\in\mathbb{S}^{N-1}$,
$$\|U'_e\|=\sup_{h\in\R^N} \frac{\|U'_e(h)\|_{L^2(\R\times\mathbb{T}^N)}}{|h|},\ \|\partial_{\xi} U'_e\|=\sup_{h\in\R^N} \frac{\|\partial_{\xi} U'_e(h)\|_{L^2(\R\times\mathbb{T}^N)}}{|h|},$$
and
$$\|\partial _{y_i} U'_e\|=\sup_{h\in\R^N} \frac{\|\partial_{y_i} U'_e(h)\|_{L^2(\R\times\mathbb{T}^N)}}{|h|}\quad (i=1,\ \cdots,\ N),\ \|U''_e\|=\sup_{(h,\rho)\in\R^N\times\R^N} \frac{\|U''_e(h)(\rho)\|_{L^2(\R\times\mathbb{T}^N)}}{|h||\rho|}.$$
Since $U_e$ is continuous with respect to $e\in\mathbb{S}^{N-1}$ and $\mathbb{S}^{N-1}$ is a compact subset of $\R^N$, one has that $\partial_{\xi} U_e$, $\partial_{y_i} U_e$ ($i=1,\ \cdots,\ N$) are also continuous with respect to $e\in\mathbb{S}^{N-1}$ and it follows from (ii) of Definition \ref{PF} that
$$\lim_{\xi\rightarrow\pm \infty} U_e(\xi,y)=0,\ 1,\text{ uniformly for $e\in\mathbb{S}^{N-1}$}.$$
This also implies that $\lim_{\xi\rightarrow \pm\infty} U'_e(h)(\xi,y)=0$ for any $h\in\R^N$, uniformly for $y\in\R^N$, $e\in\mathbb{S}^{N-1}$. Thus, $\|U'_e\|$ is bounded uniformly for $e\in\mathbb{S}^{N-1}$. Similarly, one can get that $\|\partial_{\xi} U'_e\|$, $\|\partial_{y_i} U'_e\|$ ($i=1,\ \cdots,\ N$) and $\|U''_e\|$ are bounded uniformly for $e\in\mathbb{S}^{N-1}$.


\section{Propagating speed of transition fronts}

\noindent
This section is devoted to prove Theorem \ref{th1.3}. It shows that the propagating speed of transition fronts can not be less than the infimum of the speeds of pulsating fronts and can not be larger than the supremum of the speeds of pulsating fronts. As the transition fronts concerned in homogeneous case \cite{H}, the lower bound of the propagating speed of transition fronts is related on how fast the domain in which the solution of the following Cauchy problem \eqref{vR} is close to $1$ extends and the upper bound is related on how fast the domain in which the solution of \eqref{omegaR} is close to $0$ contracts. Thus, in the following section, we prove two key-lemmas about the speed of extension or contraction.


\subsection{Two key-lemmas}
\noindent
In this section, we prove Lemma \ref{lemma3.1} and Lemma \ref{lemma3.2} below. In the sequel, we let $U_e$ be a family of pulsating fronts with normalization
$$\int_{\R^+\times\mathbb{T}^{N-1}} U_e^2(\xi,y)dyd\xi=1, \text{ for every $e\in\mathbb{S}^{N-1}$}.$$
For any $b\in\R^N\setminus\{0\}$, let $U_b$ defined in \eqref{Ub}, that is, $U_b=U_{b/|b|}$. By Theorem \ref{th1} and Theorem \ref{th2}, $U_b$ are continuous and doubly continuously Fr\'{e}chet differentiable at any $e\in \mathbb{S}^{N-1}$. We also let
$$\underline{c}=\inf_{e\in \mathbb{S}^{N-1}} c_e\quad \text{and}\quad \overline{c}=\sup_{e\in \mathbb{S}^{N-1}} c_e.$$
As we mentioned in Remark \ref{remark1.5}, one actually has that $\underline{c}=\min_{e\in\mathbb{S}^{N-1}} c_e>0$ and $\overline{c}=\max_{e\in \mathbb{S}^{N-1}} c_e<+\infty$. Fix two real numbers $\alpha$ and $\beta$ such that
$$0<\alpha<\inf_{x\in \T^N} \theta_x \le \sup_{x\in \T^N} \theta_x<\beta<1$$
where $\theta_x$ is defined in \eqref{F1} (remember that $0<\sigma<\theta_x<1-\sigma<1$ for all $x\in\mathbb{T}^N$ with $\sigma\in (0,1/2)$).

For any $R>0$, let $v_R$ and $\omega_R$ denote the solutions of the Cauchy problems
\begin{eqnarray}\label{vR}
\left\{\begin{aligned}
&(v_R)_t=\Delta v_R+f(x,v_R),\quad t>0,\ x\in\R^N,\\
&v_R(0,x)=\beta\, \text{ for } |x|<R,\  v_R(0,x)=0\,\text{ for } |x|\ge R,
\end{aligned}
\right.
\end{eqnarray}
and
\begin{eqnarray}\label{omegaR}
\left\{\begin{aligned}
&(\omega_R)_t=\Delta \omega_R+f(x,\omega_R),\quad t>0,\ x\in\R^N,\\
&\omega_R(0,x)=\alpha\, \text{ for } |x|<R,\  \omega_R(0,x)=1\,\text{ for } |x|\ge R.
\end{aligned}
\right.
\end{eqnarray}

\begin{lemma}\label{lemma3.1}
There is $R>0$ such that the following holds: for all $\varepsilon\in(0,\underline{c}]$, there is $T_{\varepsilon}>0$ such that
\be\label{ct}
v_R(t,x)\ge 1-\sigma\,\text{ for all $t\ge T_{\varepsilon}$ and $|x|\le (\underline{c}-\varepsilon)t$},
\ee
where $\sigma$ is defined in \eqref{F2}. Furthermore,
\be\label{ct1}
v_R(t,\cdot)\rightarrow 1\ \text{locally uniformly as $t\rightarrow +\infty$}.\ee
\end{lemma}

\begin{lemma}\label{lemma3.2}
For any $\varepsilon>0$, there are some real numbers $T_{\varepsilon}>0$ and $R_{\varepsilon}>0$ such that for all $R\ge R_{\varepsilon}$, the solution $\omega_R$ satisfies
\begin{align*}
\omega_R(t,x)\le \sigma\,\text{ for all $T_{\varepsilon}\le t\le \frac{R}{\bar{c}+\varepsilon}$ and $|x|\le R-(\bar{c}+\varepsilon)t$}.
\end{align*}
\end{lemma}

Lemma \ref{lemma3.1} and Lemma \ref{lemma3.2} could be viewed as analogs of Lemma 4.1 and Lemma 4.2 in \cite{H} for spatially homogeneous bistable case. However, regarding to our spatially periodic case, pulsating fronts are depending on the propagating direction $e\in\mathbb{S}^{N-1}$ and propagating speeds are different for different directions in general, which also implies the method in \cite{H} can not apply here directly.
\vskip 0.3cm

\begin{proof}[Proof of Lemma \ref{lemma3.1}]
{\it Step 1: choice of some parameters.}
Let us set $\delta=\frac{\sigma}{2},$
where $\sigma$ is defined in \eqref{F2}. Since $\lim_{\xi\rightarrow \pm\infty}U_e(\xi,y)=0$, $1$ uniformly for $y\in\R^N$ and $e\in\mathbb{S}^{N-1}$, there exists a constant $C>0$ independent of $e$ such that
$$U_e(\xi,y)\ge 1-\delta,\text{ for all $\xi\le -C$, $y\in\R^N$ and $e\in\mathbb{S}^{N-1}$},$$
and
$$
U_e(\xi,y)\le \delta,\text{ for all $\xi\ge C$, $y\in\R^N$ and $e\in\mathbb{S}^{N-1}$}.
$$
Since $\partial_{\xi} U_e$ is negative and continuous on $(\xi,y)\in\R\times\R^N$ and recalling that $\partial_{\xi} U_e$ is continuous with respect to $e\in\mathbb{S}^{N-1}$, there is a constant $k>0$  such that $-\partial_{\xi} U_e\ge k$ on $[-C,C]\times\R^N$ for all $e\in\mathbb{S}^{N-1}$. For any $\varepsilon\in(0,\underline{c}]$, let $\delta_{\varepsilon}$ such that
\be\label{delta1}
0<\delta_{\varepsilon}= \min\left(\delta,\frac{\varepsilon k}{8L}\right),
\ee
where $L:=\max_{(x,u)\in\R^N\times[0,1]}|f_u(x,u)|$. Let $C_{\varepsilon}\ge 3$ large enough such that
\be\label{C1}
\frac{N\sqrt{N}}{C_{\varepsilon}} \sup_{e\in\mathbb{S}^{N-1}}\left(3\|U'_e\|+2\|\partial_{\xi} U'_e\|+\frac{2}{N} \sum_{i=1}^N \|\partial_{y_i} U'_e\|+\frac{\sqrt{N}}{C_{\varepsilon}}\|U''_e\|\right)\le \min\left\{\frac{1}{3}\gamma\delta_{\varepsilon},\frac{\varepsilon k}{8}\right\},
\ee
where $\gamma$ is defined in \eqref{F2}, together with
\be\label{C2}
\frac{N-1}{C_{\varepsilon}}\le \frac{\varepsilon}{4}.
\ee
Similar as the definition of $C$, there exists $C'_{\varepsilon}>0$ independent of $e$ such that
$$U_e(\xi,y)\ge 1-\delta_{\varepsilon},\text{ for all $\xi\le -C'_{\varepsilon}$, $y\in\R^N$ and $e\in\mathbb{S}^{N-1}$},$$
and
$$
U_e(\xi,y)\le \delta_{\varepsilon},\text{ for all $\xi\ge C'_{\varepsilon}$, $y\in\R^N$ and $e\in\mathbb{S}^{N-1}$}.
$$

Let us now introduce an auxiliary function. It is elementary to check that there is $C^2$ function $h_{\varepsilon}:\R\rightarrow [0,1]$ such that for some $\xi_{\varepsilon}>0$,
\begin{align*}
h_{\varepsilon}(\xi)=0 \text{ for $\xi\le-\xi_{\varepsilon}-C$},\ h_{\varepsilon}(\xi)=1 \text{ for $\xi\ge -C$}\,\text{ and }\,0\le h_{\varepsilon}'(\xi)\le 1 \text{ for $\xi\in \R$}.
\end{align*}
Furthermore, we choose $\xi_{\varepsilon}$ large enough such that $h_{\varepsilon}'(\xi)$ and $h_{\varepsilon}''(\xi)$ are so small that
\be\label{h1}
2\left(|\partial_{\xi} U_{e}(\xi,y)|+|\nabla_y U_{e}(\xi,y)|\right) h_{\varepsilon}'(\xi)\le \frac{1}{3}\gamma \delta_{\varepsilon}\,\text{ for all $\xi\in\R$, $y\in\R^N$ and $e\in\mathbb{S}^{N-1}$},
\ee
and
\be\label{h2}
\delta |h_{\varepsilon}''(\xi)|\le \frac{1}{3}\gamma \delta_{\varepsilon}\,\text{ for all $\xi\in\R$}.
\ee

{\it Step 2: proof when $\underline{c}/2\le\varepsilon\le \underline{c}$.}
To do so, it is sufficient to show that Lemma \ref{lemma3.1} holds with $\varepsilon=\varepsilon_0:=\underline{c}/2>0$, for some $R>0$.

Let $\varrho_{\beta}(t,x)$ be the solution of \eqref{eq1.1} with initial condition $\varrho_{\beta}(0,x)=\beta$ for $x\in \R^N$. Since $\beta\in(\sup_{x\in \T^N} \theta_x,1)$, there holds $\varrho_{\beta}(t,x)\rightarrow 1$ as $t\rightarrow +\infty$ uniformly in $x\in\R^N$, and there is $T>0$ such that $\varrho_{\beta}(T,x)\ge 1-\delta_{\varepsilon_0}/2$ for all $x\in\R^N$. From the maximum principle, it follows that
$$0\le \varrho_{\beta}(T,x)-v_R(T,x)\le \frac{e^{LT}}{(4\pi T)^{N/2}}\int_{|y|\ge R} e^{-\frac{|x-y|^2}{4T}} dy$$
for all $R>0$ and $x\in \R^N$. Thus, if $0<B\le R$ and $|x|\le R-B$, one has that
$$0\le \varrho_{\beta}(T,x)-v_R(T,x)\le \frac{e^{LT}}{(4\pi T)^{N/2}}\int_{|z|\ge B} e^{-\frac{|z|^2}{4T}} dz.$$
Therefore, there exists a constant $B>0$ such that, for all $R\ge B$ and $|x|\le R-B$, $\varrho_{\beta}(T,x)-v_R(T,x)\le \delta_{\varepsilon_0}/2$. Then, it holds that
\begin{equation}\label{3.10}
v_R(T,x)\ge \varrho_{\beta}(T,x)-\frac{\delta_{\varepsilon_0}}{2}\ge 1-\delta_{\varepsilon_0}\ \text{for all $R\ge B$ and $|x|\le R-B$}.
\end{equation}

Let us set
\begin{align}\label{R}
R=\xi_{\varepsilon_0}+C+C_{\varepsilon_0}+C'_{\varepsilon_0}+B>B>0.
\end{align}
For the family of pulsating fronts $U_e(\xi,y)$ with $c_e$, we treat the direction $e$ as a variation $\hat{x}=\frac{x}{|x|}$ for $x\in\R^N\setminus\{0\}$ and we can get that $(U_{\hat{x}}(\xi,y),c_{\hat{x}})$ satisfies
\be\label{eq:3.13}
c_{\hat{x}} \partial_{\xi} U_{\hat{x}} +\partial_{\xi\xi} U_{\hat{x}} +2\nabla_y \partial_{\xi} U_{\hat{x}}\cdot {\hat{x}}+\Delta_y U_{\hat{x}} +f(y,U_{\hat{x}})=0,\text{ for all $(\xi,y)\in\R\times\R^N$ and $x\in\R^N\setminus\{0\}$}.
\ee
For all $(t,x)\in [T,+\infty)\times\R^N$, we set
\be\label{unv}
\underline{v}(t,x):=\max\big\{U_{\hat{x}}(\underline{\zeta}(t,x),x)h_{\varepsilon_0}(\underline{\zeta}(t,x))+(1-\delta)(1-h_{\varepsilon_0}(\underline{\zeta}(t,x))-\delta_{\varepsilon_0} ,\ 0\big\},
\ee
where
\be\label{zeta}
\underline{\zeta}(t,x)=|x|-\left(\underline{c}-\frac{\varepsilon_0}{2}\right)(t-T)-\xi_{\varepsilon_0}-C-C_{\varepsilon_0}.
\ee
Notice that, when $t\ge T$ and $|x|\le C_{\varepsilon_0}$, then $h_{\varepsilon_0}(\underline{\zeta}(t,x))=0$. Hence \eqref{unv} makes sense for $x=0$, even if $U_{\hat{x}}$ is not defined when $x=0$. Let us then check that $\underline{v}$ is a subsolution for the problem satisfied by $v_R$, for $t\ge T$ and $x\in\R^N$.

First, at the time $T$, it follows from \eqref{3.10} and the definition of $\underline{v}$ that
$$v_R(T,x)\ge 1-\delta_{\varepsilon_0}\ge \underline{v}(T,x)\,\text{ for all $|x|\le R-B$}.$$
On the other hand, if $|x|\ge R-B$, then $|x|-\xi_{\varepsilon_0}-C-C_{\varepsilon_0}\ge C'_{\varepsilon_0}$ from \eqref{R}, hence $\underline{\zeta}(t,x)\ge C'_{\varepsilon_0}>0<-C$ and $h_{\varepsilon_0}(\underline{\zeta}(t,x))=1$. From the definition of $C'_{\varepsilon_0}$ and the fact that $v_R\ge 0$ in $(0,+\infty)\times\R^N$, one has that
$$\underline{v}(T,x)\le 0\le v_R(T,x)\,\text{ for all $|x|\ge R-B$}.$$
Thus,
$$v_R(T,x)\ge \underline{v}(T,x)\,\text{ for all $x\in\R^N$}.$$

Let us now check that
\be\label{Lv}
L\underline{v}=\underline{v}_t-\Delta \underline{v}-f(x,\underline{v})\le 0,
\ee
for all $t\ge T$ and $x\in \R^N$ such that $\underline{v}>0$.
Let $(t,x)$ be any point in $[T,+\infty)\times\R^N$ such that $\underline{v}(t,x)>0$. For $(t,x)\in [T,+\infty)\times\R^N$ such that $\underline{\zeta}(t,x)< -\xi_{\varepsilon_0}-C$, one has that $h_{\varepsilon_0}(\underline{\zeta}(t,x))=0$ and
$$\underline{v}(t,x)=1-\delta-\delta_{\varepsilon_0}\ge 1-\sigma.$$
Furthermore, by continuity of $\underline{\zeta}$, this property holds in a neighborhood of such a point $(t,x)$ in $[T,+\infty)\times\R^N$.
Thus, there holds
\begin{align*}
L \underline{v}=&\underline{v}_t -\Delta \underline{v}-f(x,\underline{v})\\
=& -f(x,1-\delta-\delta_{\varepsilon_0})\le 0,
\end{align*}
from \eqref{F2} since $0<\delta_{\varepsilon_0}\le \delta=\sigma/2$.

Consider now $(t,x)\in [T,+\infty)\times\R^N$ such that $\underline{v}(t,x)>0$ and $-\xi_{\varepsilon_0}-C \le \underline{\zeta}(t,x)\le -C$. One has $|x|\ge (\underline{c}-\varepsilon_0/2)(t-T)+C_{\varepsilon_0}\ge C_{\varepsilon_0}\ge 3>0$ and
\be\label{eq3.17}
1-\delta\le U_{\hat{x}}(\underline{\zeta}(t,x))<1\,\text{ and }\,\underline{v}(t,x)\ge 1-\delta-\delta_{\varepsilon_0}\ge 1-\sigma.
\ee
After some calculations and from \eqref{eq:3.13}, there holds that
\begin{align*}
L \underline{v}=&\underline{v}_t -\Delta \underline{v}-f(x,\underline{v})\\
=& (c_{\hat{x}}-\underline{c}+\frac{\varepsilon_0}{2})\partial_{\xi} U_{\hat{x}} h_{\varepsilon_0} -(\underline{c}-\frac{\varepsilon_0}{2})[U_{\hat{x}}-(1-\delta)]h_{\varepsilon_0}' -\partial_{\xi} U_{\hat{x}}\frac{N-1}{|x|}h_{\varepsilon_0}-2\partial_{\xi} U_{\hat{x}} h_{\varepsilon_0}'\\
&-2\nabla_y U_{\hat{x}}\cdot \frac{x}{|x|} h_{\varepsilon_0}'-[U_{\hat{x}}-(1-\delta)]h_{\varepsilon_0}''-[U_{\hat{x}}-(1-\delta)]\frac{N-1}{|x|}h_{\varepsilon_0}' +f(x,U_{\hat{x}})h_{\varepsilon_0}\\
&-f(x,U_{\hat{x}}h_{\varepsilon_0}+(1-\delta)(1-h_{\varepsilon_0})-\delta_{\varepsilon_0}) -2\sum_{i=1}^{N}\partial_{\xi}U'_{\hat{x}}(\hat{x}_{x_i})(\underline{\zeta}(t,x),x)\frac{x_i}{|x|}h_{\varepsilon_0}\\ &-2\sum_{i=1}^{N}\partial_{y_i}U'_{\hat{x}}(\hat{x}_{x_i})(\underline{\zeta}(t,x),x) h_{\varepsilon_0}-2\sum_{i=1}^{N} U'_{\hat{x}}(\hat{x}_{x_i})(\underline{\zeta}(t,x),x)\frac{x_i}{|x|}h'_{\varepsilon_0} \\
&-\sum_{i=1}^{N}U''_{\hat{x}} (\hat{x}_{x_i})(\hat{x}_{x_i})(\underline{\zeta}(t,x),x) h_{\varepsilon_0} -\sum_{i=1}^{N}U'_{\hat{x}}(\hat{x}_{x_ix_i})(\underline{\zeta}(t,x),x)h_{\varepsilon_0},
\end{align*}
where $\underline{v}$ and all its derivatives are taken at $(t,x)$, $h_{\varepsilon_0}$ and all its derivatives are taken at $\underline{\zeta}(t,x)$, $U_{\hat{x}}$ and all its derivatives are taken at $(\underline{\zeta}(t,x),x)$, and for $i=1,\cdots,N$,
$$\hat{x}_{x_i}=\left(-\frac{x_1x_i}{|x|^3},\cdots,\frac{1}{|x|}-\frac{x_i^2}{|x|^3},\cdots,-\frac{x_Nx_i}{|x|^3}\right),$$
$$\hat{x}_{x_ix_i}=\left(-\frac{x_1}{|x|^3}+3\frac{x_1 x_i^2}{|x|^5},\cdots,-3\frac{x_i}{|x|^3}+3\frac{x_i^3}{|x|^5},\cdots,-\frac{x_N}{|x|^3}+3\frac{x_Nx_i^2}{|x|^5}\right).$$
Notice that $|\hat{x}_{x_i}|\le \sqrt{N}/|x|$ and $|\hat{x}_{x_ix_i}|\le \sqrt{N}/|x|$ for all $i=1,\cdots,N$ (remember that $|x|\ge C_{\varepsilon}\ge 3$). Hence,
\begin{align*}
L \underline{v}\le& \frac{\varepsilon_0}{2} \partial_{\xi} U_{\hat{x}} h_{\varepsilon_0}-\partial_{\xi} U_{\hat{x}}\frac{N-1}{|x|} h_{\varepsilon_0} -\left(2\partial_{\xi} U_{\hat{x}} +2\nabla_y U_{\hat{x}}\cdot \frac{x}{|x|}\right)h_{\varepsilon_0}'-[U_{\hat{x}}-(1-\delta)]h_{\varepsilon_0}''\\
&+f(x,U_{\hat{x}})h_{\varepsilon_0}-f(x,U_{\hat{x}}h_{\varepsilon_0}+(1-\delta)(1-h_{\varepsilon_0})-\delta_{\varepsilon_0})+2\sum_{i=1}^{N}\|\partial_{\xi} U'_{\hat{x}}\||\hat{x}_{x_i}| \\
&+2\sum_{i=1}^{N}\|\partial_{y_i}U'_{\hat{x}}\||\hat{x}_{x_i}|+2\sum_{i=1}^{N} \|U'_{\hat{x}}\||\hat{x}_{x_i}| +\sum_{i=1}^{N}\|U''_{\hat{x}}\| |\hat{x}_{x_i}|^2 +\sum_{i=1}^{N} \|U'_{\hat{x}}\||\hat{x}_{x_ix_i}|\\
\le& \frac{\varepsilon_0}{2} \partial_{\xi} U_{\hat{x}} h_{\varepsilon_0}-\partial_{\xi} U_{\hat{x}}\frac{N-1}{|x|} h_{\varepsilon_0} -\left(2\partial_{\xi} U_{\hat{x}} +2\nabla_y U_{\hat{x}}\cdot \frac{x}{|x|}\right)h_{\varepsilon_0}'-[U_{\hat{x}}-(1-\delta)]h_{\varepsilon_0}''\\
&+f(x,U_{\hat{x}})h_{\varepsilon_0}-f(x,U_{\hat{x}}h_{\varepsilon_0}+(1-\delta)(1-h_{\varepsilon_0})-\delta_{\varepsilon_0})\\
&+\frac{N\sqrt{N}}{|x|}\sup_{e\in\mathbb{S}^{N-1}}\left(3\|U'_e\|+2\|\partial_{\xi} U'_e\|+\frac{2}{N}\sum_{i=1}^N\|\partial_{y_i} U'_e\|+\frac{\sqrt{N}}{|x|}\|U''_e\|\right),
\end{align*}
since $c_{\hat{x}}\ge \underline{c}$, $\partial_{\xi} U_{\hat{x}}<0$, $0\le h_{\varepsilon_0}\le 1$, $U_{\hat{x}}\ge 1-\delta$, $0\le h_{\varepsilon_0}'\le 1$ and \eqref{eq3.17}. Since $\underline{\zeta}(t,x)\ge -\xi_{\varepsilon_0}-C$, that is, $|x|\ge (\underline{c}-\frac{\varepsilon_0}{2})(t-T)+C_{\varepsilon_0}\ge C_{\varepsilon_0}$ and from \eqref{C2}, one has that
\be\label{eq3.12}
\frac{\varepsilon_0}{2} \partial_{\xi} U_{\hat{x}} h_{\varepsilon_0}-\partial_{\xi} U_{\hat{x}}\frac{N-1}{|x|} h_{\varepsilon_0} \le \frac{\varepsilon_0}{4}\partial_{\xi} U_{\hat{x}}h_{\varepsilon_0}\le 0.
\ee
Then, from \eqref{C1}, \eqref{h1} and \eqref{h2}, it follows that
\be\label{eq3.13}
-\left(2\partial_{\xi} U_{\hat{x}} +2\nabla_y U_{\hat{x}}\cdot \frac{x}{|x|}\right)h_{\varepsilon_0}'-[U_{\hat{x}}-(1-\delta)]h_{\varepsilon_0}''\le \frac{2}{3}\gamma\delta_{\varepsilon_0},
\ee
and
\be\label{eq3.14}
\frac{N\sqrt{N}}{|x|}\sup_{e\in\mathbb{S}^{N-1}}\left(3\|U'_e\|+2\|\partial_{\xi} U'_e\|+\frac{2}{N}\sum_{i=1}^N\|\partial_{y_i} U'_e\|+\frac{\sqrt{N}}{|x|}\|U''_e\|\right)\le \frac{1}{3}\gamma\delta_{\varepsilon_0}.
\ee
On the other hand, one can calculate that
\begin{align}
&f(x,U_{\hat{x}})h_{\varepsilon_0}-f(x,U_{\hat{x}}h_{\varepsilon_0}+(1-\delta)(1-h_{\varepsilon_0})-\delta_{\varepsilon_0})\nonumber\\
=&-f(x,U_{\hat{x}})(1-h_{\varepsilon_0})+f_u(x,U_1(t,x))\left([U_{\hat{x}}-(1-\delta)](1-h_{\varepsilon_0})+\delta_{\varepsilon_0}\right).\label{eq3.21}
\end{align}
where $U_1(t,x)=U_{\hat{x}}-\theta[U_{\hat{x}}-(1-\delta)](1-h_{\varepsilon_0})-\theta\delta_{\varepsilon_0}$ for some $\theta(t,x)\in[0,1]$.
Since $U_{\hat{x}}(\underline{\zeta}(t,x))\ge 1-\delta$ for $-\xi_{\varepsilon_0}-C\le \underline{\zeta}(t,x)\le -C$ and then $U_1(t,x)\ge 1-\delta-\delta_{\varepsilon_0}\ge 1-\sigma$, it follows from \eqref{F2} and \eqref{eq3.21} that
\be\label{eq3.15}
f(x,U_{\hat{x}})h_{\varepsilon_0}-f(x,U_{\hat{x}}h_{\varepsilon_0}+(1-\delta)(1-h_{\varepsilon_0})-\delta_{\varepsilon_0})\le -\gamma\delta_{\varepsilon_0}.
\ee
Thus, it concludes from \eqref{eq3.12}-\eqref{eq3.14} and \eqref{eq3.15} that for any $(t,x)\in [T,+\infty)\times\R^N$ such that $\underline{v}(t,x)>0$ and  $-\xi_{\varepsilon_0}-C\le \underline{\zeta}(t,x)\le -C$,
$$L \underline{v}=\underline{v}_t -\Delta \underline{v}-f(x,\underline{v})\le 0.$$

For any $(t,x)\in [T,+\infty)\times\R^N$ such that $\underline{v}(t,x)>0$ and $\underline{\zeta}(t,x)> -C$, one has that
$$|x|\ge C_{\varepsilon_0},\ h_{\varepsilon_0}(\underline{\zeta}(t,x))=1,\text{ and }\underline{v}(t,x)=U_{\hat{x}}(\underline{\zeta}(t,x),x)-\delta_{\varepsilon_0},$$
and the same properties hold in a neighborhood of $(t,x)$ in $[T,+\infty)\times\R^N$.
After some calculations, there holds
\begin{align*}
L \underline{v}=&\underline{v}_t -\Delta \underline{v}-f(x,\underline{v})\\
=& (c_{\hat{x}}-\underline{c}+\frac{\varepsilon_0}{2})\partial_{\xi} U_{\hat{x}}-\partial_{\xi}U_{\hat{x}}\frac{N-1}{|x|} +f(x,U_{\hat{x}})-f(x,U_{\hat{x}}-\delta_{\varepsilon_0})\\
&-2\sum_{i=1}^{N}\partial_{\xi}U'_{\hat{x}}(\hat{x}_{x_i})(\underline{\zeta}(t,x),x) \frac{x_i}{|x|}-2\sum_{i=1}^{N}\partial_{y_i}U'_{\hat{x}}(\hat{x}_{x_i})(\underline{\zeta}(t,x),x)\\
&-\sum_{i=1}^{N}U''_{\hat{x}}(\hat{x}_{x_i})(\hat{x}_{x_i})(\underline{\zeta}(t,x),x)-\sum_{i=1}^{N}U'_{\hat{x}}(\hat{x}_{x_ix_i})(\underline{\zeta}(t,x),x)\\
\le& \frac{\varepsilon_0}{4}\partial_{\xi} U_{\hat{x}} +f(x,U_{\hat{x}})-f(x,U_{\hat{x}}-\delta_{\varepsilon_0})\\
&+\frac{N\sqrt{N}}{|x|}\sup_{e\in\mathbb{S}^{N-1}}\left(\|U'_e\|+2\|\partial_{\xi} U'_e\|+\frac{2}{N}\sum_{i=1}^N\|\partial_{y_i} U'_e\|+\frac{\sqrt{N}}{|x|}\|U''_e\|\right)
\end{align*}
from \eqref{eq3.12}. If $-C< \underline{\zeta}(t,x)\le C$, then
$$-\partial_{\xi} U_{\hat{x}}(\underline{\zeta}(t,x))\ge k\,\text{ and }\,f(x,U_{\hat{x}})-f(x,U_{\hat{x}}-\delta_{\varepsilon_0})\le L\delta_{\varepsilon_0}.$$
where $L:=\max_{(x,u)\in\R^N\times[0,1]}|f_u(x,u)|$.
From \eqref{delta1} and \eqref{C1}, one concludes that for any $(t,x)\in [T,+\infty)\times\R^N$ such that $\underline{v}(t,x)>0$ and $-C< \underline{\zeta}(t,x)\le C$,
$$L\underline{v}\le -\frac{\varepsilon_0}{4} k+\frac{\varepsilon_0 k}{8}+\frac{\varepsilon_0 k}{8}=0.$$
Finally, if $\underline{\zeta}(t,x)\ge C$, then
$$0< U_{\hat{x}}(\underline{\zeta}(t,x))\le \delta\,\text{ and }\,f(x,U_{\hat{x}})-f(x,U_{\hat{x}}-\delta_{\varepsilon_0})\le -\gamma\delta_{\varepsilon_0}.$$
From \eqref{C1} and $\partial_{\xi} U_{\hat{x}}<0$, it concludes that for any $(t,x)\in [T,+\infty)\times\R^N$ such that $\underline{v}(t,x)>0$ and $\underline{\zeta}(t,x)\ge C$,
$$L\underline{v}\le -\gamma \delta_{\varepsilon_0}+\frac{1}{3}\gamma\delta_{\varepsilon_0}\le 0.$$

As a consequence, it follows from the maximum principle that for all $t\ge T$ and $x\in \R^N$,
\begin{align}\label{3.18}
1\ge v_R(t,x)\ge \underline{v}(t,x)\ge U_{\hat{x}}(\underline{\zeta}(t,x),x)h_{\varepsilon_0}(\underline{\zeta}(t,x))+(1-\delta)(1-h_{\varepsilon_0}(\underline{\zeta}(t,x)))-\delta_{\varepsilon_0}.
\end{align}
But
$$\max_{|x|\le (\underline{c}-\varepsilon_0)t} \underline{\zeta}(t,x)\le (\underline{c}-\varepsilon_0)t-\left(\underline{c}-\frac{\varepsilon_0}{2}\right)(t-T)\rightarrow -\infty\,\text{ as $t\rightarrow +\infty$},$$
from \eqref{zeta} and the positivity of $\xi_{\varepsilon_0}$, $C$, $C_{\varepsilon_0}$. Since $h_{\varepsilon_0}(\xi)=0$ for $\xi\le -\xi_{\varepsilon_0}-C$ and \eqref{3.18}, there is $T_{\varepsilon_0}>T>0$ such that
\be\label{eq3.23}
1\ge v_R(t,x)\ge 1-\delta-\delta_{\varepsilon_0}\ge 1-\sigma\,\text{ for all $t\ge T_{\varepsilon_0}$ and $|x|\le (\underline{c}-\varepsilon_0)t$}.
\ee
Then, for any sequence $(t_n)_{n\in\mathbb{N}}$ such that $t_n\rightarrow +\infty$ as $n\rightarrow +\infty$, the sequence $v_n(t,x):=v(t+t_n,x)$ converges, up to a subsequence, to a solution $v_{\infty}(t,x)$ of \eqref{eq1.1} locally uniformly in $C^{1,2}(\R\times\R^N)$ and $v_\infty\ge 1-\sigma$ by \eqref{eq3.23}. Let $\varrho_{1-\sigma}(t,x)$ be the solution of \eqref{eq1.1} with initial condition $\varrho_{1-\sigma}(0,x)=1-\sigma$ for $x\in \R^N$. Then, $\varrho_{1-\sigma}(t,x)$ is a subsolution of the problem satisfied by $v_{\infty}(t,x)$ and $\varrho_{1-\sigma}(t,x)\rightarrow 1$ as $t\rightarrow +\infty$ since $1-\sigma>\theta_x$ for all $x\in\mathbb{T}^N$. Thus, one has that $v_{\infty}(t,x)\equiv 1$ and
\begin{align}\label{3.19}
v_R(t,x)\rightarrow 1\,\text{locally uniformly as $t\rightarrow +\infty$}.
\end{align}

{\it Step 3: proof when $0<\varepsilon\le \underline{c}$.}
We only have to show that the conclusion holds for $0<\varepsilon< \varepsilon_0$. Let now $\varepsilon$ be arbitrary in $(0,\varepsilon_0)$. We borrow the notions from Step 1 and set
\be\label{R2}
R_{\varepsilon}=\xi_{\varepsilon}+C+C_{\varepsilon}+C'_{\varepsilon}>0.
\ee
From \eqref{3.19}, there is $T_{\varepsilon}\ge T$ such that
$$v_R(T_{\varepsilon},x)\ge 1-\delta_{\varepsilon}\,\text{ for all $|x|\le R_{\varepsilon}$}.$$

We also define $\underline{v}$ and $\underline{\zeta}$ as in \eqref{unv} and \eqref{zeta} with $T$ and $\varepsilon_0$ replaced by $T_{\varepsilon}$ and $\varepsilon$. Following the same calculations as in Step 3, one gets that \eqref{Lv} holds for all $(t,x)\in [T_{\varepsilon},+\infty)\times\R^N$ such that $\underline{v}(t,x)>0$. We only have to compare $v_R$ and $\underline{v}$ at time $T_{\varepsilon}$. If $|x|\le R_{\varepsilon}$, then $v_R(t,x)\ge 1-\delta_{\varepsilon}\ge \underline{v}(T_{\varepsilon},x)$. If $|x|\ge R_{\varepsilon}$, then
$$\underline{\zeta}(T_{\varepsilon},x)=|x|-\xi_{\varepsilon}-C-C_{\varepsilon}\ge R_{\varepsilon}-\xi_{\varepsilon}-C-C_{\varepsilon}=C'_{\varepsilon}$$
from \eqref{zeta} and \eqref{R2}, whence $h_{\varepsilon}(\underline{\zeta}(T_{\varepsilon},x))=1$, $U_{\frac{x}{|x|}}(\underline{\zeta}(T_{\varepsilon},x))\le \delta_{\varepsilon}$ and $\underline{v}(T_{\varepsilon},x)=0\le v_R(T_{\varepsilon},x)$. Thus,
$$v_R(t,x)\ge \underline{v}(T_{\varepsilon},x)\,\text{ for all $x\in\R^N$}.$$
Therefore, it follows from the maximum principle that
$$v_R(t,x)\ge \underline{v}(t,x)\ge U_{\hat{x}}(\underline{\zeta}(t,x),x)h_{\varepsilon}(\underline{\zeta}(t,x))+(1-\delta)(1-h_{\varepsilon}(\underline{\zeta}(t,x))-\delta_{\varepsilon}\,\text{ for all $t\ge T_{\varepsilon}$ and $x\in\R^N$}.$$
As in Step 2, this leads to \eqref{ct} and \eqref{ct1}. This completes the proof.
\end{proof}
\vskip 0.3cm

Now we prove Lemma \ref{lemma3.2} in a similar way.
\vskip 0.3cm

\begin{proof}[Proof of Lemma \ref{lemma3.2}]
Take any $\varepsilon>0$. We borrow some notions from the proof of Lemma \ref{lemma3.1}, that is, $\delta$, $C$, $k$, $\delta_{\varepsilon}$, $C_{\varepsilon}$ and $C'_{\varepsilon}$ are defined as in Step 1 of the proof of Lemma \ref{lemma3.1}. On the other hand, the auxiliary function $h_{\varepsilon}$ needs some modification, that is, one chooses a $C^2$ function $h_{\varepsilon}:\R\rightarrow [0,1]$ such that for some $\xi_{\varepsilon}>0$,
$$
h_{\varepsilon}(\xi)=0 \text{ for $\xi\ge \xi_{\varepsilon}+C$},\ h_{\varepsilon}(\xi)=1 \text{ for $\xi\le C$}\,\text{ and }\,-1\le h_{\varepsilon}'(\xi)\le 0 \text{ for $\xi\in \R$}.
$$
Furthermore, we choose $\xi_{\varepsilon}$ large enough such that $h_{\varepsilon}'(\xi)$ and $h_{\varepsilon}''(\xi)$ are so small that
\be\label{eq3.25}
2\left(|\partial_{\xi} U_{e}(\xi,y)|+|\nabla_y U_{e}(\xi,y)|\right) |h_{\varepsilon}'(\xi)|\le \frac{1}{3}\gamma \delta_{\varepsilon}\,\text{ for all $\xi\in\R$, $y\in\R^N$ and $e\in\mathbb{S}^{N-1}$},
\ee
and
\be\label{eq3.26}
\delta |h_{\varepsilon}''(\xi)|\le \frac{1}{3}\gamma \delta_{\varepsilon}\,\text{ for all $\xi\in\R$}.
\ee

Let $\varrho_{\alpha}(t,x)$ be the solution of \eqref{eq1.1} with initial condition $\varrho_{\alpha}(0,x)=\alpha$ for $x\in \R^N$. Since $\alpha\in(0,\inf_{x\in \T^N} \theta_x)$, there holds $\varrho_{\alpha}(t,x)\rightarrow 0$ as $t\rightarrow +\infty$, and there is $\tau_{\varepsilon}>0$ such that $\varrho_{\alpha}(\tau_{\varepsilon},x)\le \delta_{\varepsilon}/2$ for all $x\in\mathbb{R}^N$. From the maximum principle, it follows that there exists $B_{\varepsilon}>0$ such that, for all $R\ge B_{\varepsilon}$ and $|x|\le R-B_{\varepsilon}$, $0\ge \varrho_{\alpha}(\tau_{\varepsilon},x)-\omega_R(\tau_{\varepsilon},x)\ge -\delta_{\varepsilon}/2$, whence
\begin{equation}\label{eq3.27}
\omega_R(\tau_{\varepsilon},x)\le \varrho_{\alpha}(\tau_{\varepsilon},x)+\frac{\delta_{\varepsilon}}{2}\le \delta_{\varepsilon}\ \text{for all $R\ge B_{\varepsilon}$ and $|x|\le R-B_{\varepsilon}$}.
\end{equation}

We choose $T_{\varepsilon}\ge \tau_{\varepsilon}$ such that
\be\label{eq3.28}
\frac{\varepsilon t}{2}\ge C+\xi_{\varepsilon}+B_{\varepsilon}+C'_{\varepsilon}\,\text{ for all $t\ge T_{\varepsilon}$},
\ee
and $R_{\varepsilon}>0$ such that
\be\label{eq3.29}
R_{\varepsilon}\ge \max\left(B_{\varepsilon},(\bar{c}+\varepsilon)T_{\varepsilon}\right)\,\text{ and }\, \frac{\varepsilon R_{\varepsilon}}{2(\bar{c}+\varepsilon)}\ge B_{\varepsilon}+C+\xi_{\varepsilon}+C'_{\varepsilon}+C_{\varepsilon}.
\ee

In the sequel, let $R$ be an arbitrary real number such that $R\ge R_{\varepsilon}$. For the family of pulsating fronts $U_e(\xi,y)$ with $c_e$, we treat the direction $e$ as a variation $\tilde{x}=-\frac{x}{|x|}$ for $x\in\R^N\setminus\{0\}$ and we can get that $(U_{\tilde{x}}(\xi,y),c_{\tilde{x}})$ satisfies
\be\label{eq3.30}
c_{\tilde{x}} \partial_{\xi} U_{\tilde{x}} +\partial_{\xi\xi} U_{\tilde{x}} +2\nabla_y \partial_{\xi} U_{\tilde{x}}\cdot \tilde{x}+\Delta_y U_{\tilde{x}} +f(y,U_{\tilde{x}})=0,\text{ for all $(\xi,y)\in\R\times\R^N$ and $x\in\R^N\setminus\{0\}$}.\ee
Set
$$E:=\left[\tau_{\varepsilon},\frac{R}{\bar{c}+\varepsilon}\right]\times\R^N.$$
For all $(t,x)\in E$, we set
\begin{equation}\label{3.31}
\bar{\omega}(t,x):=\min\big\{U_{\tilde{x}}(\overline{\zeta}(t,x),x)h_{\varepsilon}(\overline{\zeta}(t,x))+\delta(1-h_{\varepsilon}(\overline{\zeta}(t,x))+\delta_{\varepsilon},\ 1\big\},
\end{equation}
where
\be\label{eq3.31}
\overline{\zeta}(t,x)=-|x|-\left(\bar{c}+\frac{\varepsilon}{2}\right)(t-\tau_{\varepsilon})+R-B_{\varepsilon}-C'_{\varepsilon}.
\ee
Notice that, when $\tau_{\varepsilon}\le t\le R/(\overline{c}+\varepsilon)$ and $|x|\le C_{\varepsilon}$, then $\overline{\zeta}(t,x)\ge C+\xi_{\varepsilon}$ by \eqref{eq3.29} and $h_{\varepsilon}(\overline{\zeta}(t,x))=0$. Hence \eqref{3.31} makes sense for $x=0$, even if $U_{\tilde{x}}$ is not defined when $x=0$.
Let us check that $\bar{\omega}$ is a supersolution for the problem satisfied by $\omega_R$, in the set $E$.

At the time $\tau_{\varepsilon}$, one can follow from \eqref{eq3.27}, \eqref{eq3.29} and the definition of $\bar{\omega}$ that
$$\omega_R(\tau_{\varepsilon},x)\le \delta_{\varepsilon}\le \bar{\omega}(\tau_{\varepsilon},x)\,\text{ for all $|x|\le R-B_{\varepsilon}$}.$$
On the other hand, if $|x|\ge R-B_{\varepsilon}$, then $\overline{\zeta}(\tau_{\varepsilon},x)=-|x|+R-B_{\varepsilon}-C'_{\varepsilon}\le -C'_{\varepsilon}<0<C$, hence $h(\overline{\zeta}(\tau_{\varepsilon},x))=1$. From the definition of $C'_{\varepsilon}$ and the fact that $\omega_R\le 1$ in $(0,+\infty)\times\R^N$, one has that
$$\bar{\omega}(\tau_{\varepsilon},x)=1\ge \omega_R(\tau_{\varepsilon},x),\text{ for all $|x|\ge R-B_{\varepsilon}$}.$$
Thus,
$$\overline{\omega}(\tau_{\varepsilon},x)\ge \omega_R(\tau_{\varepsilon},x), \text{ for all $x\in\R^N$}.$$

Let us now check that
$$L\bar{\omega}=\bar{\omega}_t-\Delta \bar{\omega} -f(x,\bar{\omega})\ge 0$$
for all $(t,x)\in E$ such that $\bar{\omega}(t,x)<1$. This will be sufficient to ensure that $\overline{\omega}$ is a supersolution.
Let $(t,x)$ be any point in $E$ such that $\bar{\omega}(t,x)<1$. For $(t,x)\in E$ such that $\overline{\zeta}(t,x)> C+\xi_{\varepsilon}$, one has that $h_{\varepsilon}(\overline{\zeta}(t,x))=0$ and
$$\bar{\omega}(t,x)=\delta+\delta_{\varepsilon}\le \sigma.$$
Thus, there holds
\begin{align*}
L \bar{\omega}=\bar{\omega}_t -\Delta \bar{\omega}-f(x,\bar{\omega})= -f(x,\delta+\delta_{\varepsilon})\ge 0,
\end{align*}
from \eqref{F2} since $0<\delta_{\varepsilon}\le \delta\le \sigma/2$.

Consider now $(t,x)\in E$ such that $\overline{\omega}(t,x)< 1$ and $C \le \overline{\zeta}(t,x)\le C+\xi_{\varepsilon}$. One has $|x|\ge -(\overline{c}+\varepsilon/2)(t-\tau_{\varepsilon})+R-B_{\varepsilon}-C'_{\varepsilon}-C-\xi_{\varepsilon}\ge C_{\varepsilon}\ge 3>0$ by \eqref{eq3.29} and
$$0< U_{\tilde{x}}(\overline{\zeta}(t,x))\le \delta\,\text{ and }\,\bar{\omega}(t,x)\le \delta+\delta_{\varepsilon}\le \sigma.$$
After some calculations and from \eqref{eq3.30}, there holds that
\begin{align*}
L \bar{\omega}=&\bar{\omega}_t -\Delta \bar{\omega}-f(x,\bar{\omega})\\
=& (c_{\tilde{x}}-\bar{c}-\frac{\varepsilon}{2})\partial_{\xi} U_{\tilde{x}} h_{\varepsilon} -(\bar{c}+\frac{\varepsilon}{2})(U_{\tilde{x}}-\delta)h_{\varepsilon}' +\partial_{\xi} U_{\tilde{x}}\frac{N-1}{|x|}h_{\varepsilon}-2\partial_{\xi} U_{\tilde{x}} h_{\varepsilon}'\\
&+2\nabla_y U_{\tilde{x}}\cdot \frac{x}{|x|} h_{\varepsilon}' -(U_{\tilde{x}}-\delta)h_{\varepsilon}'' +(U_{\tilde{x}}-\delta)\frac{N-1}{|x|}h_{\varepsilon}'+f(x,U_{\tilde{x}})h_{\varepsilon}\\
&-f(x,U_{\tilde{x}}h_{\varepsilon}+\delta(1-h_{\varepsilon})+\delta_{\varepsilon}) +2\sum_{i=1}^{N}\partial_{\xi}U'_{\tilde{x}}(\tilde{x}_{x_i})(\overline{\zeta}(t,x),x) \frac{x_i}{|x|}h_{\varepsilon}\\
&-2\sum_{i=1}^{N}\partial_{y_i}U'_{\tilde{x}}(\tilde{x}_{x_i})(\overline{\zeta}(t,x),x) h_{\varepsilon} +2\sum_{i=1}^{N} U'_{\tilde{x}}(\tilde{x}_{x_i})(\overline{\zeta}(t,x),x) \frac{x_i}{|x|}h'_{\varepsilon}\\
&-\sum_{i=1}^{N}U''_{\tilde{x}}(\tilde{x}_{x_i})(\tilde{x}_{x_i})(\overline{\zeta}(t,x),x) h_{\varepsilon}-\sum_{i=1}^{N}U'_{\tilde{x}}(\tilde{x}_{x_ix_i})(\overline{\zeta}(t,x),x)h_{\varepsilon}
\end{align*}
where $\overline{\omega}$ and all its derivatives are taken at $(t,x)$, $h_{\varepsilon_0}$ and all its derivatives are taken at $\overline{\zeta}(t,x)$, $U_{\tilde{x}}$ and all its derivatives are taken at $(\overline{\zeta}(t,x),x)$, and
$$\tilde{x}_{x_i}=\left(\frac{x_1x_i}{|x|^3},\cdots,-\frac{1}{|x|}+\frac{x_i^2}{|x|^3},\cdots,\frac{x_Nx_i}{|x|^3}\right),$$
$$\tilde{x}_{x_ix_i}=\left(\frac{x_1}{|x|^3}-3\frac{x_1 x_i^2}{|x|^5},\cdots,3\frac{x_i}{|x|^3}-3\frac{x_i^3}{|x|^5},\cdots,\frac{x_N}{|x|^3}-3\frac{x_Nx_i^2}{|x|^5}\right).$$
Notice that $|\tilde{x}_{x_i}|\le \sqrt{N}/|x|$ and $|\tilde{x}_{x_ix_i}|\le \sqrt{N}/|x|$ for all $i=1,\cdots,N$ (remember that $|x|\ge C_{\varepsilon}\ge 3$). Hence,
\begin{align*}
L\overline{\omega}\ge& -\frac{\varepsilon}{2} \partial_{\xi} U_{\tilde{x}} h_{\varepsilon}-(\bar{c}+\frac{\varepsilon}{2})(U_{\tilde{x}}-\delta)h_{\varepsilon}'+\partial_{\xi} U_{\tilde{x}}\frac{N-1}{|x|} h_{\varepsilon} +\left(-2\partial_{\xi} U_{\tilde{x}} +2\nabla_y U_{\tilde{x}}\cdot \frac{x}{|x|}\right)h_{\varepsilon}'\\
&-(U_{\tilde{x}}-\delta)h_{\varepsilon}''+f(x,U_{\tilde{x}})h_{\varepsilon}-f(x,U_{\tilde{x}}h_{\varepsilon} +\delta(1-h_{\varepsilon})+\delta_{\varepsilon})\\
&-\frac{N\sqrt{N}}{|x|} \sup_{e\in\mathbb{S}^{N-1}}\left(3\|U'_{e}\|+2\|\partial_{\xi} U'_{e}\|+\frac{2}{N}\sum_{i=1}^N\|\partial_{y_i} U'_{e}\|+\|U''_{e}\|\right),
\end{align*}
since $c_{\tilde{x}}\le \bar{c}$, $0<U_{\tilde{x}}\le \delta$, $h_{\varepsilon}\le 1$, and $-1\le h'_{\varepsilon}\le 0$. From $|x|\ge C_{\varepsilon}$ and \eqref{C2}, one has that
\be\label{eq3.32}
-\frac{\varepsilon}{2} \partial_{\xi} U_{\tilde{x}} h_{\varepsilon}+\partial_{\xi} U_{\tilde{x}}\frac{N-1}{|x|} h_{\varepsilon} \ge -\frac{\varepsilon}{4}\partial_{\xi} U_{\tilde{x}}h_{\varepsilon}\ge 0.
\ee
Then, from \eqref{C1}, \eqref{eq3.25} and \eqref{eq3.26}, it follows that
\be\label{eq3.33}
\left(-2\partial_{\xi} U_{\tilde{x}} +2\nabla_y U_{\tilde{x}}\cdot \frac{x}{|x|}\right)h_{\varepsilon}'-(U_{\tilde{x}}-\delta)h''_{\varepsilon}\ge -\frac{2}{3}\gamma\delta_{\varepsilon},
\ee
and
\be\label{eq3.34}
-\frac{N\sqrt{N}}{|x|} \sup_{e\in\mathbb{S}^{N-1}}\left(3\|U'_{e}\|+2\|\partial_{\xi} U'_{e}\|+\frac{2}{N}\sum_{i=1}^N\|\partial_{y_i} U'_{e}\|+\|U''_{e}\|\right)\ge -\frac{1}{3}\gamma\delta_{\varepsilon}.
\ee
On the other hand, one can calculate that,
\begin{align}
&f(x,U_{\tilde{x}})h_{\varepsilon}-f(x,U_{\tilde{x}}h_{\varepsilon}+\delta(1-h_{\varepsilon})+\delta_{\varepsilon})\nonumber\\
=&-f(x,U_{\tilde{x}})(1-h_{\varepsilon})-f_u(x,U_2(t,x)) \left((\delta-U_{\tilde{x}})(1-h_{\varepsilon})+\delta_{\varepsilon}\right).\label{eq3.35}
\end{align}
where $U_2(t,x)=U_{\tilde{x}}-\theta(U_{\tilde{x}}-\delta)(1-h_{\varepsilon})+\theta\delta_{\varepsilon}$ for some $\theta(t,x)\in[0,1]$.
Since $U_{\tilde{x}}(\overline{\zeta}(t,x))\le \delta$ for $C\le \overline{\zeta}(t,x)\le \xi_{\varepsilon}+C$ and then $U_2(t,x)\le \delta+\delta_{\varepsilon}\le \sigma$, it follows from \eqref{F2} and \eqref{eq3.35} that
\be\label{eq3.36}
f(x,U_{\tilde{x}})h_{\varepsilon}-f(x,U_{\tilde{x}}h_{\varepsilon}+\delta(1-h_{\varepsilon})+\delta_{\varepsilon})\ge \gamma\delta_{\varepsilon}.
\ee
Thus, it concludes from \eqref{eq3.32}-\eqref{eq3.34} and \eqref{eq3.36} that for any $(t,x)\in E$ such that $\overline{\omega}(t,x)<1$ and $C\le \overline{\zeta}(t,x)\le \xi_{\varepsilon}+C$,
$$
L \underline{v}=\underline{v}_t -\Delta \underline{v}-f(x,\underline{v})\ge 0.
$$

For any $(t,x)\in E$ such that $\overline{\omega}(t,x)<1$ and $\overline{\zeta}(t,x)< C$, one has that
$$|x|\ge C_{\varepsilon},\ h_{\varepsilon}(\overline{\zeta}(t,x))=1,\text{ and }\bar{\omega}(t,x)=U_{\tilde{x}}(\overline{\zeta}(t,x),x)+\delta_{\varepsilon_0},$$
and the same properties hold in a neighborhood of $(t,x)$ in $E$.
After some calculations, there holds
\begin{align*}
L \bar{\omega}=&\bar{\omega}_t -\Delta \bar{\omega}-f(x,\bar{\omega})\\
=& (c_{\tilde{x}}-\bar{c}-\frac{\varepsilon}{2})\partial_{\xi} U_{\tilde{x}} +\partial_{\xi} U_{\tilde{x}}\frac{N-1}{|x|} +f(x,U_{\tilde{x}})-f(x,U_{\tilde{x}}+\delta_{\varepsilon})\\
&+2\sum_{i=1}^{N}\partial_{\xi}U'_{\tilde{x}}(\tilde{x}_{x_i})(\overline{\zeta}(t,x),x) \frac{x_i}{|x|}-2\sum_{i=1}^{N}\partial_{y_i}U'_{\tilde{x}}(\tilde{x}_{x_i})(\overline{\zeta}(t,x),x)\\
&-\sum_{i=1}^{N}U''_{\tilde{x}}(\tilde{x}_{x_i})(\tilde{x}_{x_i})(\overline{\zeta}(t,x),x) -\sum_{i=1}^{N}U'_{\tilde{x}}(\tilde{x}_{x_ix_i})(\overline{\zeta}(t,x),x)\\
\ge& -\frac{\varepsilon}{4}\partial_{\xi} U_{\tilde{x}} +f(x,U_{\tilde{x}})-f(x,U_{\tilde{x}}+\delta_{\varepsilon})\\
&-\frac{N\sqrt{N}}{|x|} \sup_{e\in\mathbb{S}^{N-1}} \left(\|U'_{e}\|+2\|\partial_{\xi} U'_e\|+\frac{2}{N}\sum_{i=1}^N\|\partial_{y_i} U'_e\|+\|U''_e\|\right)
\end{align*}
from \eqref{eq3.32}. If $-C\le \overline{\zeta}(t,x)< C$, then
$$-\partial_{\xi} U_{\tilde{x}}(\overline{\zeta}(t,x))\ge k\,\text{ and }\,f(x,U_{\tilde{x}})-f(x,U_{\tilde{x}}+\delta_{\varepsilon})\ge - L\delta_{\varepsilon}.$$
From \eqref{delta1} and \eqref{C1}, one concludes that for any $(t,x)\in E$ such that $\overline{\omega}(t,x)<1$ and $-C\le \overline{\zeta}(t,x)< C$,
$$L \bar{\omega}\ge \frac{\varepsilon}{4} k-\frac{\varepsilon k}{8}-\frac{\varepsilon k}{8}=0.$$
Finally, if $\overline{\zeta}(t,x)\le -C$, then
$$1-\delta\le U_{\tilde{x}}(\overline{\zeta}(t,x))< 1\,\text{ and }\,f(x,U_{\tilde{x}})-f(x,U_{\tilde{x}}+\delta_{\varepsilon})\ge \gamma\delta_{\varepsilon}.$$
From \eqref{C1} and $\partial_{\xi} U_{\frac{x}{|x|}}<0$, it concludes that for $(t,x)\in E$ such that $\overline{\omega}(t,x)<1$ and $\overline{\zeta}(t,x)\le -C$
$$L \bar{\omega}\ge \gamma \delta_{\varepsilon}-\frac{1}{3}\gamma\delta_{\varepsilon}\ge 0.$$

As a conclusion, it follows from the maximum principle that for all $(t,x)\in [\tau_{\varepsilon}, R\slash(\bar{c}+\varepsilon)]\times\R^N$,
$$0\le \omega_R(t,x)\le \bar{\omega}(t,x)\le U_{\tilde{x}}(\overline{\zeta}(t,x),x)h_{\varepsilon}(\overline{\zeta}(t,x))+\delta(1-h_{\varepsilon}(\overline{\zeta}(t,x))+\delta_{\varepsilon}.$$
For all $T_{\varepsilon}\le t\le R/(\bar{c}+\varepsilon)$ and $|x|\le R-(\bar{c}+\varepsilon)t$, it follows from \eqref{eq3.28} that
\begin{align*}
\overline{\zeta}(t,x)=-|x|-\left(\bar{c}+\frac{\varepsilon}{2}\right)(t-\tau_{\varepsilon})+R-B_{\varepsilon}-C'_{\varepsilon}\ge \frac{\varepsilon}{2} t + \left(\bar{c}+ \frac{\varepsilon}{2}\right) \tau_{\varepsilon} -B_{\varepsilon}-C'_{\varepsilon}\ge C+\xi_{\varepsilon}.
\end{align*}
Thus, $h_{\varepsilon}(\overline{\zeta}(t,x))=0$ and
$$\omega_R(t,x)\le \bar{\omega}(t,x)=\delta+\delta_{\varepsilon}\le \sigma.$$
This completes the proof.
\end{proof}


\subsection{Proof of Theorem \ref{th1.3}}
\noindent
This section is devoted to prove
$$\inf_{e\in\mathbb{S}^{N-1}} c_e\,\le\,\liminf_{|t-s|\rightarrow+\infty}\frac{d(\Gamma_t,\Gamma_s)}{|t-s|}\,\le\,\limsup_{|t-s|\rightarrow+\infty} \frac{d(\Gamma_t,\Gamma_s)}{|t-s|}\le\,\sup_{e\in\mathbb{S}^{N-1}} c_e.$$
Once we have the two-key lemmas, Lemma \ref{lemma3.1} and Lemma \ref{lemma3.2}, one can follow the proof of \cite[Theorem 2.7]{H} to get Theorem \ref{th1.3}. But we still sketch it for completeness. Since the second inequality is obvious, we only prove the first one and the third one in the following.

{\it Step 1: proof of the first inequality.}
Let $\varepsilon>0$ be arbitrary positive real number. Let us assume by contradiction that
\be\label{eq3.37}
\liminf_{|t-s|\rightarrow+\infty} \frac{d(\Gamma_t,\Gamma_s)}{|t-s|}<\underline{c}-2\varepsilon.
\ee
where $\underline{c}=\inf_{e\in\mathbb{S}^{N-1}} c_e$ (notice that this yields especially $0<\varepsilon\le \underline{c}/2<\underline{c}$).
There are two sequences $(t_k)_{k\in\N}$ and $(s_k)_{k\in\N}$ in $\R$ such that $|t_k-s_k|\rightarrow +\infty$ as $k\rightarrow +\infty$ and
$$d(\Gamma_{t_k},\Gamma_{s_k})<(\underline{c}-2\varepsilon)|t_k-s_k|\,\text{ for all $k\in\N$}.$$
We assume that $t_k<s_k$ for all $k\in\N$ without loss of generality. By definition of distance $d(\Gamma_{t_k},\Gamma_{s_k})$, there are then two sequences $(x_k)_{k\in\N}$ and $(z_k)_{k\in\N}$ in $\R^N$ such that
$$x_k\in \Gamma_{t_k},\ z_k\in\Gamma_{s_k}\,\text{ and }\,|x_k-z_k|<(\underline{c}-2\varepsilon)(s_k-t_k)\,\text{ for all $k\in\N$}.$$

From Definition \ref{TF}, there is $M\ge 0$ such that
\begin{eqnarray*}
\left\{\begin{aligned}
&\forall t\in \mathbb{R},\ \ \forall x\in \Omega_t^+, \ \ \left(d(x,\Gamma_t)\geq M\right)\Rightarrow \left(u(t,x)\geq 1-\sigma\right)\!,\\
&\forall t\in \mathbb{R},\ \ \forall x\in \Omega_t^-, \ \ \left(d(x,\Gamma_t)\geq M\right)\Rightarrow \left(u(t,x)\leq \sigma\right)\!.
\end{aligned}
\right.
\end{eqnarray*}
Let $R>0$ such that Lemma \ref{lemma3.1} holds true with $v_R$ defined for $\beta=1-\sigma$ and $R$. From \eqref{eq1.5}, there are $r_{R+M}$ and $y^+_k$ such that
$$y^+_k\in\Omega_{t_k}^+,\ |x_k-y^+_k|\le r_{R+M}\,\text{ and } d(y^+_k,\Gamma_{t_k})\ge R+M,$$
and $r_{M}$ and $y^-_k$ such that
$$y^-_k\in\Omega_{s_k}^-,\ |z_k-y^-_k|\le r_{M}\,\text{ and } d(y^-_k,\Gamma_{s_k})\ge M.$$
These imply that $B(y^+_k,R)\subset \Omega_{t_k}^+$, $d(B(y^+_k,R),\Gamma_{t_k})\ge M$ and $u(s_k,y^-_{k})\le \sigma$. Thus, $u(t_k,x)\ge 1-\sigma$ for all $x\in B(y^+_k,R)$. Therefore, $u(t_k,x)\ge v_R(0,x-y^+_k)$ for all $x\in\R^N$ and it follows from the maximum principle that
$$u(t,x)\ge v_R(t-t_k,x-y^+_k)\,\text{ for all $k\in\N$, $t>t_k$ and $x\in\R^N$}.$$
Then, by Lemma \ref{lemma3.1}, one has that, for every $k\in\N$,
\be\label{eq3.38}
u(t,x)\ge 1-\sigma\,\text{ for all $t\ge t_k+T_{\varepsilon}$ and $|x-y^+_k|\le (\underline{c}-\varepsilon)(t-t_k)$}.
\ee
Since $s_k-t_k\rightarrow +\infty$ as $k\rightarrow +\infty$, there is $k$ large enough such that $s_k-t_k\ge T_{\varepsilon} \text{ and } \varepsilon(s_k-t_k)\ge r_{R+M}+r_M$.
Since $|y^+_{k}-x_k|\le r_{R+M}$ and $|x_k-z_k|<(\underline{c}-2\varepsilon)(s_k-t_k) $, it follows that $|y^+_{k}-z_{k}|\le r_{R+M}+(\underline{c}-2\varepsilon)(s_k-t_k)$. On the other hand, from $|z_k-y^-_k|\le r_{M}$, we get that $|y^+_{k}-y^-_{k}|\le r_{R+M}+(\underline{c}-2\varepsilon)(s_k-t_k)+r_M\le (\underline{c}-\varepsilon)(s_k-t_k)$. Thus, from \eqref{eq3.38}, $u(s_k,y^-_{k})\ge 1-\sigma$ which contradicts that $u(s_k,y^-_{k})\le \sigma$.

{\it Step 2: proof of the third inequality.}
Let $\varepsilon>0$ be arbitrary positive real number. Let us assume by contradiction that
\be\label{eq3.40}
\limsup_{|t-s|\rightarrow+\infty} \frac{d(\Gamma_t,\Gamma_s)}{|t-s|}>\bar{c}+3\varepsilon.
\ee
where $\bar{c}=\sup_{e\in\mathbb{S}^{N-1}} c_e$.
Then, there are two sequences $(t_k)_{k\in\N}$ and $(s_k)_{k\in\N}$ in $\R$ that $|t_k-s_k|\rightarrow +\infty$ as $k\rightarrow+\infty$ and
$$d(\Gamma_{t_k},\Gamma_{s_k})>(\bar{c}+3\varepsilon)|t_k-s_k|\,\text{ for all $k\in\N$}.$$
We assume that $t_k<s_k$ for all $k\in\N$ without loss of generality. For each $k\in\N$, take a point $z_k$ on $\Gamma_{s_k}$. There are two sequences $(y^{\pm}_k)_{k\in\N}$ such that
$$y^{\pm}_k\in\Omega^{\pm}_{s_k},\ |z_k-y^{\pm}_k|\le r_M,\ d(y^{\pm}_k,\Gamma_{s_k})\ge M.$$
It implies that
\be\label{eq3.41}
u(s_k,y^-_{k})\le \sigma< 1-\sigma\le u(s_k,y^+_{k}).
\ee
On the other hand, since $d(z_k,\Gamma_{t_k})>(\bar{c}+3\varepsilon)(s_k-t_k)>0$, there holds
$$\text{either }\,B(z_k,(\bar{c}+3\varepsilon)(s_k-t_k))\subset\Omega^+_{t_k}\,\text{ or }B(z_k,(\bar{c}+3\varepsilon)(s_k-t_k))\subset\Omega^-_{t_k}.$$

Assume by contradiction that, up to a subsequence,
$$B(z_k,(\bar{c}+3\varepsilon)(s_k-t_k))\subset\Omega^+_{t_k},$$
for all $k\in\mathbb{N}$. Since $s_k-t_k\rightarrow +\infty$ as $k\rightarrow+\infty$, one has $B(z_k,R)\subset \Omega^+_{t_k}$ with $d(B(z_k,R),\Gamma_{t_k})\ge M$ for all $k$ large enough. Thus, $u(t_k,x)\ge 1-\sigma\,\text{ for all $x\in B(z_k,R)$}$.
Then, $u(t_k,x)\ge v_R(0,x-z_k)$ for all $x\in\R^N$ and
$$u(t,x)\ge v_R(t-t_k,x-z_k)\,\text{ for all $k$ large enough, $t>t_k$ and $x\in\R^N$},$$
from the maximum principle. From Lemma \ref{lemma3.1}, for $\varepsilon'=\underline{c}/2$, there is $T_{\varepsilon'}>0$ such that, for all $k$ large enough,
$$u(t,x)\ge v_R(t-t_k,x-z_k)\ge 1-\sigma\,\text{ for all $t\ge t_k+T_{\varepsilon'}$ and $|x-z_k|\le (\underline{c}-\varepsilon')(t-t_k)=\frac{\underline{c}}{2}(t-t_k)$}.$$
Since $\underline{c}>0$ and $s_k-t_k\rightarrow +\infty$, one has $s_k-t_k\ge T_{\varepsilon'}$ and $|y^-_{k}-z_k|\le r_M\le \underline{c}/2(s_k-t_k)$ for all $k$ large enough. Therefore, $u(s_k,y^-_{k})\ge 1-\sigma$ for all $k$ large enough which contradicts  \eqref{eq3.41}.

Hence, for all $k$ large enough,
$$B(z_k,(\bar{c}+3\varepsilon)(s_k-t_k))\subset\Omega^-_{t_k}.$$
Since $s_k-t_k\rightarrow +\infty$ as $k\rightarrow +\infty$, it follows that $B(z_k,(\bar{c}+2\varepsilon)(s_k-t_k))\subset\Omega^-_{t_k}$ and $d(B(z_k,(\bar{c}+2\varepsilon)(s_k-t_k)),\Gamma_{t_k})\ge M$. Hence, $u(t_k,x)\le \sigma$ for all $x\in B(z_k,(\bar{c}+2\varepsilon)(s_k-t_k))$ and then $u(t_k,x)\le \omega_{(\bar{c}+2\varepsilon)(s_k-t_k)}(0,x-z_k)$ for all $x\in\R^N$ where $\omega_R$ is defined in \eqref{omegaR} with $\alpha=\sigma$. From the maximum principle, it follows that
$$u(t,x)\le \omega_{(\bar{c}+2\varepsilon)(s_k-t_k)}(t-t_k,x-z_k)\,\text{ for all $k$ large enough, $t>t_k$ and $x\in\R^N$}.$$
Since $(\overline{c}+2\varepsilon)(s_k-t_k)\rightarrow +\infty$ as $k\rightarrow +\infty$, if follows from Lemma \ref{lemma3.2} that, for all $k$ large enough,
$$u(t,x)\le \omega_{(\bar{c}+2\varepsilon)(s_k-t_k)}(t-t_k,x-z_k)\le \sigma,$$
for all $T_{\varepsilon}\le t-t_k\le (\bar{c}+2\varepsilon)(s_k-t_k)\slash(\bar{c}+\varepsilon)$ and $|x-z_k|\le (\bar{c}+2\varepsilon)(s_k-t_k)-(\bar{c}+\varepsilon)(t-t_k)$, where $T_{\varepsilon}>0$ is given in Lemma \ref{lemma3.2}. Since for all $k$ large enough, $T_{\varepsilon}\le s_k-t_k\le (\bar{c}+2\varepsilon)(s_k-t_k)\slash(\bar{c}+\varepsilon)$ and $|y^+_{k}-z_k|\le r_M\le (\bar{c}+2\varepsilon)(s_k-t_k)-(\bar{c}+\varepsilon)(s_k-t_k)$, it follows that
$$u(s_k,y^+_{k})\le \sigma.$$
which contradicts \eqref{eq3.41}.

In conclusion, we have shown that \eqref{eq3.37} and \eqref{eq3.40} are impossible for arbitrary $\varepsilon>0$. The proof of Theorem \ref{th1.3} thereby complete.\hfill $\Box$
\vskip 0.3cm

\textbf{Acknowledgement.} The author is grateful to Professor Fran\c{c}ois Hamel for his patient discussions and helpful suggestions.


\end{document}